\documentclass[12pt,epsfig,amsfonts]{amsart}
\usepackage{amsmath,amsthm,amssymb,amscd,epsfig,color}
\usepackage{graphicx}

\usepackage{url}
\usepackage[all]{xy}
\usepackage{mathrsfs}
\usepackage{ulem}
\usepackage{pb-diagram} 
\usepackage{fancybox}

\newtheorem*{dense-cor}{Theorem B}
\newtheorem*{code-thm}{Theorem C}
\usepackage{url}

\newtheorem{prop}{Proposition}[section]
\newtheorem{lemma}[prop]{Lemma}

\newtheorem{thm}[prop]{Theorem}
\newtheorem{cor}[prop]{Corollary}

\theoremstyle{definition}
\newtheorem{definition}[prop]{Definition}

\theoremstyle{remark}
\newtheorem{remark}[prop]{Remark}

\numberwithin{equation}{section}

\begin{document}

\author{Hiroki Takahasi}

\address{Keio Institute of Pure and Applied Sciences (KiPAS), Department of Mathematics,
Keio University, Yokohama,
223-8522, JAPAN} 
\email{hiroki@math.keio.ac.jp}


\thanks{}

\date{}

 \title[The Recurrence spectrum beyond specification]
{The Recurrence spectrum for dynamical systems\\ 
beyond specification}

 \maketitle

\begin{abstract}
We introduce {\it (W')-specification} 
in terms of language decompositions of subshifts, 
and show that any recurrence set 
of a subshift with this property has full Hausdorff dimension. 
Our main result applies to a wide class of subshifts  without specification, 
such as all $S$-gap shifts, some coded shifts, and the coding space of any transitive piecewise monotonic interval map with positive entropy.
Further, for a wide class of piecewise expanding interval maps 
we show that any recurrence set has full Hausdorff dimension.
\end{abstract}

 \section{Introduction}

 An important concept in dynamics is that of recurrence. In order to analyze the recurrence of a given initial point, a natural approach is to consider the growth rate of the first return times of its orbit into smaller and smaller sets that contain the initial point.   For a measurable map $T\colon X\to X$,
the first return time of a point $x\in X$ into a set $A\subset X$ is defined by \[\tau_A(x)=\inf\{n\in\mathbb N\colon T^n(x)\in A\}.\]
Let
$\xi$ be a finite partition of $X$ into measurable sets. For each $n\in\mathbb N$  let $\xi_n=\{B_1\cap T^{-1}(B_2)\cdots \cap T^{-n+1}(B_{n})\colon B_i\in\xi\text{ for } i=1,\ldots,n \}$. 
Let $\xi_n(x)$ denote the element of $\xi_n$  
that contains $x$. 
Ornstein and Weiss \cite{OW93} proved that for any $T$-invariant ergodic probability measure $\mu$, we have
\[\lim_{n\to\infty}\frac{\log\tau_{\xi_n(x)}(x)}{n}=h_\mu(T,\xi)\ \text{ $\mu$-a.e.,} \]
where $h_\mu(T,\xi)$ denotes the entropy of $(T,\mu)$ with respect to the partition $\xi$.
The measure-theoretic entropy $h_\mu(T)$ of $\mu$ with respect to $T$ is given by $h_\mu(T)=\sup_\xi h_\mu(T,\xi)$, where the supremum ranges over all finite partitions of $X$ into measurable sets. If $\xi$ is a generator then $h_\mu(T,\xi)=h_\mu(T)$ holds.
For related results on recurrence in differentiable dynamical systems, see e.g., \cite{BS01,S06,STV02}.

In order to analyze a wide variety of recurrence that is not captured by a single ergodic measure, 
for $a$, $b\in[0,\infty]$ with $a\leq b$ we consider {\it a recurrence set} 
\[R_{T,\xi}(a,b)=\left\{x\in X\colon\liminf_{n\to\infty}\frac{\log\tau_{\xi_n(x)}(x)}{n}=a\ \text{ and } \ \limsup_{n\to\infty}\frac{\log\tau_{\xi_n(x)}(x)}{n}=b\right\}.\]
If $X$ is a metric space, then let $\dim_{\rm H}$ denote the Hausdorff dimension on $X$,
and define {\it a recurrence spectrum}
 \[(a,b)\mapsto \dim_{\rm H}R_{T,\xi}(a,b).\]
 This spectrum was first investigated by Feng and Wu \cite{FW01}  for the one-sided full shift. Interestingly, they
 proved that any recurrence set with respect to the partition into $1$-cylinders has full Hausdorff dimension. 
Later on,
some natural variants of the recurrence set were considered and 
analogous statements were established.
For conformal repellers, Saussol and Wu \cite{SW03} considered the first return time into a ball of radius $r>0$ instead of cylinders, normalized by $-\log r$ instead of $n$.
Independently from \cite{SW03} and with different techniques,  Olsen \cite{Ols04} obtained stronger results for some conformal iterated function systems. 
  Kim and Li \cite{KL16} replaced the normalizing coefficients $n$ by an increasing function $\varphi(n)$ of $n$, and established a zero-one law for the Hausdorff dimension of the recurrence sets of one-sided full shifts.

The aim of this paper is to
 extend Feng and Wu's result \cite{FW01} to a wide class of subshifts. 
 For some dynamical systems 
 with specification (see section~\ref{sp-sec} for various notions of specification), Lau and Shu \cite{LS08} considered recurrence sets and an associated spectrum replacing cylinders by Bowen balls and Hausdorff dimension by topological entropy. However, the notion of specification used in \cite{LS08} is  strong. 
 Inspired by the works of 
 Climenhaga and Thompson \cite{CT12,CT13}, 
in this paper we introduce {\it (W')-specification}, and show that any recurrence set of a subshift with (W')-specification has full Hausdorff dimension (theorem~\ref{mainthm}). We also establish an analogous statement for a wide class of interval maps (theorem~\ref{monotone-thm}).

 \subsection{Recurrence spectrum for subshifts}
 Throughout this paper 
 we assume $m\geq2$ is a fixed integer.
 Let $\Sigma_m$ denote the 
 full shift on $m$ symbols, namely 
 $\Sigma_m=\{0,\ldots,m-1\}^{\mathbb N}$ in the one-sided case and $\Sigma_m=\{0,\ldots,m-1\}^{\mathbb Z}$ in the two-sided case.
We introduce a metric $d$ on $\Sigma_m$ by
\begin{equation}\label{metric-d}d(x,y)=\exp(-\min\{|i|\colon x_i\neq y_i\} ),\end{equation}
for two distinct elements $x=(x_i)_{i\in\mathbb G}$, $y=(y_i)_{i\in\mathbb G}$,
where 
$\mathbb G$ is $\mathbb N$ or $\mathbb Z$ according as $\Sigma_m$ is one-sided or two-sided.
Let $\sigma$ denote the left shift  acting on $\Sigma_m$: $(\sigma^kx)_i=x_{i+k}$ for $k\in\mathbb N\cup\{0\}$. A $\sigma$-invariant closed subset of $\Sigma_m$ is called {\it a subshift}.

We introduce an {\it empty word} $\emptyset$ by the rules
  $\emptyset w=w\emptyset=w$ for any word $w$ from $\{0,\ldots,m-1\}$ 
  and $\emptyset\emptyset=\emptyset$.
  Let $|w|$ denote 
  the word length of a word $w$, 
  and set the word length of the empty word to be $0$. 
   {\it A language} of a subshift $\Sigma$,
denoted by $\mathcal L(\Sigma)$, is the collection of the empty word and all words 
that appear in some elements of $\Sigma$. 
For each $n\in\mathbb N\cup\{0\}$, let $\mathcal L_n(\Sigma)$ denote the collection of elements of $\mathcal L(\Sigma)$ with word length $n$.  
We set 
$\#\mathcal L_0(\Sigma)=1$ for convenience.   
For $w_1\cdots w_n\in\mathcal L_n(\Sigma)$, 
 define {\it an $n$-cylinder}
\[[w_1\cdots w_n]=
\{x\in\Sigma\colon x_i=w_{i}\text{ for }i=1,\ldots,n\}.\]
We will consider the recurrence sets in $\Sigma$ with respect to the natural partition 
 \[\xi=\{[0],\ldots,[m-1]\}.\]
 Note that $\xi$ is a generator. 
 If $\Sigma$ is one-sided, then
for $x\in\Sigma$ and $n\in\mathbb N$,
 $\tau_{\xi_{n}(x)}(x)$ is the smallest $k\in\mathbb N$ such that $\sigma^kx$ enters the ball of radius $e^{-n-1}$ about $x$
 with respect to the metric in \eqref{metric-d}.

  For a subshift $\Sigma$,
  we consider subsets $\mathcal C^{\rm p}$, $\mathcal G$, $\mathcal C^{\rm s}$ of $\mathcal L(\Sigma)$ such that
$\mathcal L(\Sigma)=\mathcal C^{\rm p}\mathcal G\mathcal C^{\rm s}$. This means that every word in $\mathcal L(\Sigma)$
can be written as a concatenation of 
elements of $\mathcal C^{\rm p}$, $\mathcal G$, $\mathcal C^{\rm s}$. For a subset $\mathcal E$ of $\mathcal L(\Sigma)$ and $n\in\mathbb N\cup\{0\}$, let $\mathcal E_n=\mathcal E\cap\mathcal L_n(\Sigma)$ and put
\[h(\mathcal E )=\limsup_{n\to\infty}\frac{\log \#\mathcal E_n}{n}.\] 
Let $h_{\rm top}(\Sigma)$ denote the topological entropy of $\Sigma$,
namely $h_{\rm top}(\Sigma)=h(\mathcal L(\Sigma))$.

\begin{thm}\label{mainthm}
Let $\Sigma$ be a subshift that
has a language decomposition $\mathcal L(\Sigma)=\mathcal C^{\rm p}\mathcal G\mathcal C^{\rm s}$ such that $\mathcal G$ has (W')-specification and $h(\mathcal C^{\rm p}\cup\mathcal C^{\rm s})<h_{\rm top}(\Sigma)$.
For all $a,b\in[0,\infty]$ with $a\leq b$, we have
\[\dim_{\rm H}R_{\sigma,\xi}(a,b)=\dim_{\rm H}\Sigma.\]
\end{thm}
We say a subshift
 $\Sigma$ has {\it the specification property} if there is an integer $t\geq0$ such that for all $u,v\in\mathcal L(\Sigma)$ there is $w\in\mathcal L(\Sigma)$ such that $uwv\in\mathcal L(\Sigma)$ and $|w|= t$.  
This property 
roughly means that one can glue together a collection of orbit segments to form one orbit.
Points in the recurrence sets with prescribed properties may be constructed inductively, by replicating initial pieces of orbit segments which are constructed in previous steps. Hence, the specification property plays an important role in the analysis of recurrence sets.

The equality in theorem~\ref{mainthm} is known for the full shift \cite{FW01}, and for subshifts with specification \cite{LS08} such as transitive subshifts of finite type and sofic shifts \cite{LM}. 
For subshifts without specification, it was known only for the beta shifts \cite{Par60,Sch97} 
as a consequence of the result of Ban and Liu
\cite[theorem~1.2]{Ban}. The argument in \cite{Ban} relies on a special approximation property within the parametrized family of beta shifts that cannot be generalized much. 

The definition of (W')-specification in theorem~\ref{mainthm} is given in section~\ref{sp-sec}. This property is 
a variant of
(W)-specification introduced by Climenhaga and Thompson \cite{CT12,CT13} for the purpose of developing an ergodic theory for dynamical systems beyond specification.
As in lemma~\ref{w-implication}, 
(W')-specification implies (W)-specification.
Many of the presently known subshifts that satisfy (W)-specification also satisfy (W')-specification.
Below is a list of three classes of subshifts that satisfy the assumption of theorem~\ref{mainthm}.

\subsubsection{Coding spaces of interval maps}
We say $T\colon [0,1]\to [0,1]$
is  {\it a piecewise monotonic map} if
there exist pairwise disjoint,
finitely many non-degenerate 
subintervals $I_0,\ldots,I_{m-1}$ of $[0,1]$ 
such that $\bigcup_{j=0}^{m-1} I_j=[0,1]$, and the restriction
$T|_{{\rm int}(I_j)}$
is strictly monotone and continuous for every $0\leq j\leq m-1$.

Let $T$ be a piecewise monotonic map and let \[\xi_T=\{I_0,\ldots,I_{m-1}\}.\] 
By following orbits of $T$ over the  partition $\xi_T$, one can code points in $[0,1]$ into symbol sequences in a subshift in $\Sigma_m$. This subshift is denoted by $\Sigma_T$, and called {\it the coding space} of $T$ (see section~\ref{Markov-sec}).
We say $T$ is {\it transitive} if there is $x\in [0,1]$ such that $\{T^n(x)\colon n\ge 0\}$ is dense in $[0,1]$. 
If $T$ is  transitive then $\xi_T$ is a genearator. If moreover $T$ has positive topological entropy, then
for any $\varepsilon>0$ there is a decomposition
$\mathcal L(\Sigma_T)=\mathcal C^{\rm p}\mathcal G\mathcal C^{\rm s}$ such that $\mathcal G$ has  (W')-specification and $h(\mathcal C^{\rm p}\cup\mathcal C^{\rm s})<\varepsilon$ (see proposition~\ref{CT-dec}). 
The coding space of any transitive piecewise monotonic map with positive topological entropy satisfies the assumption of Theorem~\ref{mainthm}.

\subsubsection{Coded shifts}
Coded shifts \cite{CT12,KSW23,LM} are two-sided shifts that may satisfy the assumption of Theorem~\ref{mainthm}. 
Well-studied coded shifts are $S$-gap shifts \cite{LM} and the Dyck shift \cite{Kri74}.
It can be read out from \cite[section~4]{CT12} that any coded shift $\Sigma$ has a natural language decomposition $\mathcal L(\Sigma)=\mathcal C^{\rm p}\mathcal G\mathcal C^{\rm s}$ such that $\mathcal G$ has (W')-specification with gap size $0$. Therefore, any coded shift $\Sigma$ with
$h(\mathcal C^{\rm p}\cup\mathcal C^{\rm s})<h_{\rm top}(\Sigma)$ satisfies the assumption of Theorem~\ref{mainthm}. This inequality holds for any $S$-gap shift, while it does not hold for
the Dyck shift.

\subsubsection{Subshifts in \cite[theorem~1.1]{C18}}\label{climen}
Let $v=v_1\cdots v_{|v|}$ and $w$ be words from $\{0,\ldots,m-1\}$ that are not the empty word. 
We say $w$  is  {\it a prefix} of $v$ if  
$w=v_{1}\cdots v_{|v|-i}$ for some $0\leq i\leq|v|-1$. 
We say $w$ is {\it a suffix} of $v$ if
$w=v_i\cdots v_{|v|}$ for some $1\leq i\leq|v|$.
For each $n\in\{1,\ldots,|v|\}$,  let $v|^{\rm p}_n$ and $v|^{\rm s}_n$ denote the prefix and suffix of $v$ of length $n$ respectively.
Let $\Sigma$ be a subshift for which
there are subsets $\mathcal C^{\rm p}$, $\mathcal G$, $\mathcal C^{\rm s}$ of $\mathcal L(\Sigma)$ such that:
\begin{itemize}
\item[(I)] there is $t\in\mathbb N$ such that for all $v$, $w\in\mathcal G\setminus\{\emptyset\}$
there is $u\in\mathcal L(\Sigma)$
with $|u|\leq t$ such that
$v|^{\rm s}_ku(w|^{\rm p}_n)\in\mathcal G$
whenever $k,n\in\mathbb N$ satisfy $1\leq k\leq |v|$, $1\leq n\leq|w|$ 
and $v|^{\rm s}_k\in\mathcal G$, $w|^{\rm p}_n\in\mathcal G$;

\item[(II)] 
$h(\mathcal C^{\rm p}\cup \mathcal C^{\rm s}\cup(\mathcal L(\Sigma)\setminus\mathcal C^{\rm p}\mathcal G\mathcal C^{\rm s}))<h_{\rm top}(\Sigma)$;
\item[(III)] there is $M\in\mathbb N$ such that
if $u,v,w\in\mathcal L(\Sigma)$ satisfy $|v|\geq M$, $uvw\in\mathcal L(\Sigma)$, $uv$, $vw\in\mathcal G$, then $v$, $uvw\in\mathcal G$. 
\end{itemize}
Under this assumption, Climenhaga \cite[theorem~1.1]{C18} obtained
a collection of results on the existence and uniqueness of equilibrium states, and  statistical properties of the equilibrium states.
Without loss of generality we may assume $\mathcal L(\Sigma)=\mathcal C^{\rm p}\mathcal G\mathcal C^{\rm s}$. Indeed, if $\mathcal L(\Sigma)\setminus\mathcal C^{\rm p}\mathcal G\mathcal C^{\rm s}$ is non-empty then replacing $\mathcal C^{\rm p}$ by $\mathcal C^{\rm p}\cup(\mathcal L(\Sigma)\setminus\mathcal C^{\rm p}\mathcal G\mathcal C^{\rm s})$ we get $\mathcal L(\Sigma)=\mathcal C^{\rm p}\mathcal G\mathcal C^{\rm s}$. 
Then (II) becomes $h(\mathcal C^{\rm p}\cup\mathcal C^{\rm s})<h_{\rm top}(\Sigma)$, and
 (I) implies 
that $\mathcal G$ has (W')-specification with gap size $t$. Overall, any subshift considered in \cite[theorem~1.1]{C18} satisfies the assumption of
theorem~\ref{mainthm}.

The Hausdorff dimension of a subshift $\Sigma$ with respect to the metric in \eqref{metric-d} equals
$h_{\rm top}(\Sigma)$ or $2h_{\rm top}(\Sigma)$ according as
 $\Sigma$ is one-sided or two-sided.
To prove theorem~\ref{mainthm} it suffices to  show that for all $a,b\in[0,\infty]$ with $a\leq b$, 
\begin{equation}\label{lower}\dim_{\rm H}R_{\sigma,\xi}(a,b)\geq\begin{cases}\vspace{1mm}\displaystyle{h_{\rm top}(\Sigma)}&\text{ if $\Sigma$ is one-sided,}\\
\displaystyle{2h_{\rm top}(\Sigma)  }&\text{ if $\Sigma$ is two-sided. }\end{cases}\end{equation}
   We prove \eqref{lower} in section~2 by constructing for each $(a,b)$ Moran fractals of large Hausdorff dimension contained in $R_{\sigma,\xi}(a,b)$.
Moran fractals are sets that are constructed by iterative procedure using a countable number of parameters, see e.g., \cite{Pes}. The Hausdorff dimension of Moran fractals can be estimated from below with the mass distribution principle \cite{Fal14}.

The construction 
of our Moran fractals breaks into three steps.
In the first step (section~\ref{seed-sec}),
 for all sufficiently large integer $k$ we construct {\it a seed set} $F_k$ that is independent of $(a,b)$, consists of non-recurrent points and has large Hausdorff dimension. 
In the second step (section~\ref{const-sec}), we modify the seed sets into subsets of the recurrence set $R_{\sigma,\xi}(a,b)$.
In the final step (section~\ref{moran-sec}) we refine these subsets to obtain our Moran fractals.
Although inspired by the construction of Feng and Wu for the full shift \cite{FW01},  
our construction of Moran fractals
  necessarily gets more involved as developed below.
  
In the first step we construct the seed set $F_k$ by representing an ergodic measure of large entropy with a collection of words and concatenating them. Due to the existence of forbidden words, it is not possible to fully concatenate a given collection of words, and (W)-specification \cite{CT12,CT13} plays a crucial role. A similar construction was successfully undertaken in \cite{STY24} for different purposes.

(W')-specification plays a crucial role in the second step.  
  In section~\ref{sp-sec} we briefly explain how it is used and why (W)-specification does not suffice.
To keep the reasonable length of this paper,
we only give a proof of theorem~\ref{mainthm} for one-sided shifts.
 A proof for two-sided shifts is completely analogous.

\subsection{Recurrence spectrum for piecewise expanding interval maps}
We also establish a statement analogous to theorem~\ref{mainthm} on the recurrence sets for a wide class of interval maps. 
We say a piecewise monotonic map $T$ is {\it of class $C^r$} $(r\geq1)$ if for each $I\in\xi_T$ there is an open interval $J$ containing $I$ such that  $T|_I$ can be extended to a $C^r$ function on $J$. We say a piecewise monotonic map $T$ is {\it piecewise expanding} if it is of class $C^1$ and 
$\inf|T'|>1$. 
If $T$ is a piecewise expanding, then $\xi_T$ is a generator.

\begin{thm}\label{monotone-thm}Let $T\colon[0,1]\to[0,1]$ be a transitive piecewise expanding map of class $C^2$.
For all $a,b\in[0,\infty]$ with $a\leq b$, we have
\[\dim_{\rm H }R_{T,\xi_T}(a,b)=1.\]
\end{thm}

Many maps in theorem~\ref{monotone-thm} have  discontinuities, and such maps rarely have specification \cite[theorem~1.5]{Buz97}. 
On the equality in theorem~\ref{monotone-thm}
for maps without specification, 
the only previous result we are aware of is again that of Ban and Liu
\cite[theorem~1.2]{Ban} on beta transformations
\cite{R57}.

 As a testbed of
 theorem~\ref{monotone-thm}, we take {\it the alpha-beta transformations} \cite{Par64}. These are piecewise monotonic expanding maps
$T_{\alpha,\beta}\colon[0,1]\to[0,1]$ given by 
\[T_{\alpha,\beta}(x)=\begin{cases}\beta x+\alpha-\lfloor  \beta x+\alpha\rfloor&\text{ for }x\in[0,1),\\
\lim_{x\nearrow1}T_{\alpha,\beta}(x)&\text{ for }x=1,\end{cases}\]
where $\alpha\in[0,1)$ and $\beta>1$.
The iteration of $T_{\alpha,\beta}$ generates a representation of real numbers.
For each $x\in[0,1)$
write $e_{\alpha,\beta,1}(x)=\lfloor\beta x+\alpha\rfloor$, and
$e_{\alpha,\beta,k+1}(x)=e_{\alpha,\beta,1}(T_{\alpha,\beta}^k(x))$
for $k\in\mathbb N$. 
For all $x\in[0,1)$ we have
\[x=\frac{e_{\alpha,\beta,1}(x)-\alpha+T_{\alpha,\beta}(x)}{\beta},\] and thus
\[T_{\alpha,\beta}(x)=\frac{e_{\alpha,\beta,1}(T_{\alpha,\beta}(x))-\alpha+T_{\alpha,\beta}^2(x)}{\beta}.\]
Plugging the last equality into the previous one gives
\[x=\frac{e_{\alpha,\beta,1}(x)-\alpha}{\beta}+\frac{e_{\alpha,\beta,2}(x)-\alpha}{\beta^2}+\frac{T_{\alpha,\beta}^2(x)}{\beta^2}.\]
Repeating this procedure we obtain {\it the $(\alpha,\beta)$-expansion} of $x$:
\[x=\sum_{n=1}^\infty\frac{e_{\alpha,\beta,n}(x)-\alpha}{\beta^n}.\]

If $m\in\mathbb N$, $m\geq2$ and $m-1-\alpha<\beta\leq m-\alpha$, then
the integers $e_{\alpha,\beta,k}(x)$, $k\in\mathbb N$ are in $\{0,\ldots,m-1\}$,  and form a sequence that encodes the $T_{\alpha,\beta}$-orbit of $x$ 
into a sequence in 
$\Sigma_m$. 
Define pairwise disjoint subintervals $I_j$, $j=0,\ldots,m-1$ 
of $[0,1)$ by
 \[\begin{split}I_0=\left[0,\frac{1-\alpha}{\beta}\right)\ \text{ and }\  
 I_{m-1}
 =\left[\frac{m-1-\alpha}{\beta},1\right),\end{split}\]
 and if $m\geq3$ then
\[I_j=\left[\frac{j-\alpha}{\beta},\frac{j+1-\alpha}{\beta}\right) \
 \text{ for }j=1,\ldots,m-2.\]
 Notice that $e_{\alpha,\beta,n}(x)=j$ if and only if $T_{\alpha,\beta}^{n-1}(x)\in I_j$.
Define
\[\tau_{\alpha,\beta,n}(x)=\inf\{k\in\mathbb N\colon e_{\alpha,\beta,i+k}(x)=e_{\alpha,\beta,i}(x)\text{ for }i=1,\ldots,n\}.\] 
As a corollary to theorem~\ref{monotone-thm} we obtain the following statement.

\begin{cor}\label{cor1}
If $(\alpha,\beta)\in[0,1)\times(1,\infty)$ and $T_{\alpha,\beta}$ is transitive,  
then for all $a,b\in[0,\infty]$ with $a\leq b$, we have
\[\dim_{\rm H}\left\{x\in [0,1)\colon\liminf_{n\to\infty}\frac{\log\tau_{\alpha,\beta,n}(x)}{n}=a\ \text{ and }\ \limsup_{n\to\infty}\frac{\log\tau_{\alpha,\beta,n}(x)}{n}=b\right\}=1.\]
\end{cor}

For $\alpha=0$, 
the equality in corollary~\ref{cor1} follows from
\cite{FW01,Ols04,SW03} for 
$\beta\in\mathbb N$, and from \cite{Ban} for any
$\beta>1$.
 The set of  $(\alpha,\beta)\in(0,1)\times(1,\infty)$ for which the coding space of $T_{\alpha,\beta}$ has the periodic specification property is of Lebesgue measure zero \cite{Buz97}, of Hausdorff dimension two \cite{T25}.
For $(\alpha,\beta)$ in this set, the equality in corollary~\ref{cor1} follows from \cite{LS08}.
It is not hard to show that 
if $0\leq\alpha<1$ and $\beta\geq2$ then $T_{\alpha,\beta}$ is transitive (see proposition~\ref{trans-prop} in appendix~A). When $\alpha\leq 2-\beta$, $T_{\alpha,\beta}$ is not always transitive \cite{G90,Pal79}.

Theorem~\ref{monotone-thm} is not a direct translation of theorem~\ref{mainthm}.
Our main tool is {\it a Markov diagram} introduced by Hofbauer \cite{H2}.
In section~3, 
we show that the coding space of any map in theorem~\ref{monotone-thm} satisfies the assumption in theorem~\ref{mainthm}. This allows us to immediately transfer the Moran fractals in the coding space constructed in section~2 to the interval $[0,1]$.
 When applying the mass distribution principle to the transferred Moran fractals, the difficulty is that the lengths of cylinders in the coding space are not well comparable to the Euclidean lengths of the corresponding intervals in $[0,1]$.
 This is due to the lack of specification or Markovianness, see e.g., \cite[proposition~2.6]{LW08} for the case of the beta transformations. 
 We have managed to obtain a nice lower bound on the lengths of intervals, see lemma~\ref{diam-lem}.

The rest of this paper consists of two sections. 
Section~2 is devoted to the analysis of recurrence sets for subshifts in theorem~\ref{mainthm}. Section~3 is devoted to the analysis of recurrence sets for interval maps in theorem~\ref{monotone-thm}.

\section{Dimension of recurrence sets for subshifts}
This section is devoted to the analysis of recurrence sets for subshifts in theorem~\ref{mainthm}.
In section~\ref{sp-sec} we introduce (W')-specification, and investigate its relationship with (W)-specification in \cite{CT12,CT13}.
In section~\ref{compt-sec} we state one lemma the proof of which is deferred to appendix~A.
In section~\ref{seed-sec} we construct seed sets. In section~\ref{const-sec} we modify the seed sets into recurrence sets, and in section~\ref{moran-sec} 
construct Moran fractals refining these subsets. 
In section~\ref{end-sec} we complete the proof of theorem~\ref{mainthm}.
\subsection{Various notions of specification}\label{sp-sec}
We say a subshift $\Sigma$
has   {\it the periodic specification property}  if there is an integer $t\geq0$ 
such that 
for every integer $k\geq2$ and all $v^1,\ldots,v^k\in\mathcal L(\Sigma)$, there are $w^1,\ldots, w^{k}\in\mathcal L(\Sigma)$ such that
$v^1w^1v^2w^2
\cdots v^kw^k\in\mathcal L(\Sigma)$ and $|w^i|= t$ for $i=1,\ldots,k$, and the cylinder
$[v^1w^1v^2w^2
\cdots v^kw^k]$ contains a periodic point
of period $|v^1w^1v^2w^2\cdots v^kw^k|$.
If we further replace $|w^i|=t$ by $|w^i|\leq t$, we say $\mathcal L(\Sigma)$ has 
{\it (W)-specification} (\cite[p.792 second remark]{CT12}).

The periodic specification property is a shift space version of the definition of specification introduced in \cite{Bow71}. See also \cite{DGS,KH}.
The specification property and the periodic specification property are equivalent \cite[section~2]{KLO16}. 
The (W)-specification for $\mathcal L(\Sigma)$ is a shift space version of the definition of specification in \cite{LS08}.

A natural way to relax the specification property is to require that connected words be chosen from a proper subset of $\mathcal L(\Sigma)$ as follows.

\begin{definition}[(W)-specification \cite{CT12,CT13}]
Let $\Sigma$ be a subshift.
A proper subset $\mathcal G$ of $\mathcal L(\Sigma)$ has   {\it (W)-specification} if there is an integer $t\geq0$, called a {\it gap size}, 
such that  for all integer $k\geq2$ and all $v^1,\ldots,v^k\in\mathcal G$, there are $w^1,\ldots, w^{k-1}\in\mathcal L(\Sigma)$ such that $v^1w^1v^2w^2
\cdots w^{k-1}v^k\in\mathcal L(\Sigma)$ and $|w^i|\leq t$ for $i=1,\ldots,k-1$.
\end{definition} 

We will construct
points in designated recurrence sets by constructing longer and longer words inductively,  replicating at each step prefixes of words which are constructed in the previous step. See section~\ref{const-sec} for details. Clearly, (W)-specification cannot be used for this construction. Below we introduce another specification that is appropriate for this inductive construction.

\begin{definition}[(W')-specification]
\label{def-strong}
Let $\Sigma$ be a subshift.
A subset $\mathcal G$ of $\mathcal L(\Sigma)$ 
containing the empty word has {\it (W')-specification} if there exists an integer $t\geq0$, called a {\it gap size}, such that
for all integers $j,k\geq2$ and all $v^1,\ldots,v^{j+k}\in\mathcal G$,  all
   $w^1,\ldots,w^{j+k-1}\in\mathcal L(\Sigma)$ 
 such that $u:=v^1w^1v^2w^2\cdots w^jv^j\in\mathcal L(\Sigma)$,  $v:=v^{j+1}w^{j+1}v^{j+2}w^{j+2}\cdots w^{j+k}v^{j+k}\in \mathcal L(\Sigma)$ 
  and $|w^i|\leq t$ for $i=1,\ldots,j+k-1$,
 there is $w\in\mathcal L(\Sigma)$ such that  
 $uwv\in\mathcal L(\Sigma)$ and $|w|\leq t$.
\end{definition}
In other words, the (W')-specification for $\mathcal G$ with gap size $t$ requires that any pair $u$, $v$ of concatenations of elements of $\mathcal G$ by words of lengths not exceeding $t$ be connected by a word $w$ of length not exceeding $t$.

\begin{lemma}\label{w-implication}
Let $\Sigma$ be a subshift. If a subset $\mathcal G$ of $\mathcal L(\Sigma)$ containing the empty word has (W')-specification with gap size $t$, then $\mathcal G$ has (W)-specification with gap size $t$.\end{lemma}
\begin{proof}Let $k\geq2$ be an integer and let $v^1,\ldots,v^k\in\mathcal G$.  
Since $v^1=v^1\emptyset\emptyset$, $v^2=v^2\emptyset\emptyset$, and $\mathcal G$ contains the empty word, by the (W')-specification with gap size $t$ there is $w^1\in\mathcal L(\Sigma)$ such that $v^1w^1v^2\in\mathcal L(\Sigma)$ and $|w^1|\leq t$. If $k=2$ then we are done. Otherwise, by the (W)-specification with gap size $t$, 
there is $w^2\in\mathcal L(\Sigma)$ such that $v^1w^1v^2w^2v^3\in\mathcal L(\Sigma)$ and $|w^2|\leq t$. If $k=3$ then we are done. Otherwise we repeat the same argument. This verifies the (W)-specification for $\mathcal G$ with gap size $t$. \end{proof}

\subsection{Computational lemma}\label{compt-sec}
The next lemma will be used in the construction of Moran fractals contained in each recurrence set.
\begin{lemma}\label{computational}For all $a,b\in[0,\infty]$ with $a\leq b$, there exist 
sequences $(\ell_p)_{p=1}^\infty$,  $(\gamma_p)_{p=1}^\infty$ of positive reals such that: \begin{itemize}\item[(a)]
$\lim_{p\to\infty}\gamma_{2p+1}=a$ and $\lim_{p\to\infty}\gamma_{2p}=b$;
\item[(b)] $\lim_{p\to\infty}(\ell_{p+1}-\ell_p)= \infty$;
\item[(c)]
$\lim_{p\to\infty}\ell_p/e^{\gamma_p\ell_p}=0$;

\item[(d)] $
\limsup_{p\to\infty}\gamma_{p+1}\ell_{p+1}/\ell_p\leq b$;
\item[(e)]
$\gamma_{p} \ell_p+1\leq\gamma_{p+1} \ell_{p+1}$
for all $p\in\mathbb N$;

\item[(f)]for any constant $c\geq0$,  $\lim_{p\to\infty}(cp+\sum_{j=1}^{p} \ell_j)/e^{\gamma_p\ell_p}=0.$
\end{itemize}

\end{lemma}
Since a proof of lemma~\ref{computational} is purely computational,
we defer it to appendix~\ref{compute-s}. Nevertheless,
it is an integral component of the proofs of our main results.

\subsection{Construction of seed sets}\label{seed-sec}
Let $\Sigma$ be a subshift.
Let $\mathcal M(\Sigma,\sigma)$ denote the set of $\sigma$-invariant Borel probability measures on $\Sigma$. For
a function $\varphi\colon\Sigma\to\mathbb R$
and an integer $k\geq1$, write $S_k\varphi$ for the sum $\sum_{i=0}^{k-1}\varphi\circ\sigma^i$.
Let $v$, $w=w_1\cdots w_n$ be words from $\{0,\ldots,m-1\}$. We say $v$ is {\it a subword} of $w$ if there exist $i,j\in\{1,\ldots n\}$ such that $i\leq j$ and $w_i\cdots w_j=v$.
The next proposition is central to the construction of seed sets.
\begin{prop}\label{nov***}Let $\Sigma$ be a subshift  
that has a language decomposition $\mathcal L(\Sigma)=\mathcal C^{\rm p}\mathcal G\mathcal C^{\rm s}$ with $h(\mathcal C^{\rm p}\cup\mathcal C^{\rm s})<h_{\rm top}(\Sigma)$.
Let $\mu\in\mathcal M(\Sigma,\sigma)$ be ergodic and satisfy $h_{\mu}(\sigma)>h(\mathcal C^{\rm p}\cup\mathcal C^{\rm s})$. 
Let 
$\varphi\colon\Sigma\to\mathbb R$ be a continuous function.
For any sufficiently small $\varepsilon\in(0,h_\mu(\sigma))$
and any $M\geq1$, 
there exist an integer $k\geq M$, a subset $\mathcal Q_k$ 
of $\mathcal G_k$ and $v^*\in\mathcal G_k\setminus\mathcal Q_k$ such that:
\begin{itemize}
\item[(a)] 
$\#\mathcal Q_{k}\geq \exp\left((h_\mu(\sigma)-\varepsilon)k\right)$;
\item[(b)] for all $v\in\mathcal Q_k$ and all $x\in[v]$, 
\[\left|\frac{S_k\varphi(x) }{k}- \int \varphi d\mu\right|<\varepsilon;
\]

\item[(c)] if $t\in\mathbb N\cup\{0\}$, $w\in\mathcal L(\Sigma)$ satisfy $|w|\leq t<k/3$, then 
for all $u,v\in\mathcal Q_k$, $v^*$ is not a subword of $uwv$.
\end{itemize}
\end{prop}

\begin{remark}\label{rem-varphi}The function $\varphi$ in proposition~\ref{nov***} is actually not needed in the proof of theorem~\ref{mainthm}. It is needed only in the proof of theorem~\ref{monotone-thm}. \end{remark}

We finalize the construction of seed sets subject to proposition~\ref{nov***}. Let $\Sigma$ be a subshift  
that has a language decomposition $\mathcal L(\Sigma)=\mathcal C^{\rm p}\mathcal G\mathcal C^{\rm s}$ with $h(\mathcal C^{\rm p}\cup\mathcal C^{\rm s})<h_{\rm top}(\Sigma)$ and $\mathcal G$ has (W')-specification with gap size $t$. 
Let $\mu\in\mathcal M(\Sigma,\sigma)$ be ergodic with $h_\mu(\sigma)>h(\mathcal C^{\rm p}\cup\mathcal C^{\rm s})$ and let $\varphi\colon\Sigma\to\mathbb R$ be  continuous. Let $\varepsilon\in(0,h_\mu(\sigma))$ be sufficiently small and 
let $k>3t$ be a sufficiently large integer for which there exist $\mathcal Q_k\subset\mathcal G_k$ and $v^*\in\mathcal G_k\setminus\mathcal Q_k$ as in proposition~\ref{nov***}.
Define 
\[F_k=\left\{v^* w^1v^1w^2v^2\cdots\in\Sigma\colon v^i\in \mathcal Q_{k},\  w^i\in\mathcal L(\Sigma),\ |w^i|\leq t\text{ for }i\in\mathbb N \right\},\]
and call $F_k$ {\it a seed set}.
In other words, elements of $F_k$ are
points in $[v^*]$ that are obtained by
concatenating elements of $\mathcal Q_k$
by words of length not exceeding $t$.
By lemma~\ref{w-implication}, $\mathcal G$ has (W)-specification with gap size $t$.
By virtue of (W)-specification, $F_k$ is non-empty.
proposition~\ref{nov***}(c) implies that any point in $F_k$ is non-recurrent, which is important in subsequent constructions.

\begin{proof}[Proof of proposition~\ref{nov***}]

Let $\Sigma$ be a subshift  
that has a language decomposition $\mathcal L(\Sigma)=\mathcal C^{\rm p}\mathcal G\mathcal C^{\rm s}$ such that $h(\mathcal C^{\rm p}\cup\mathcal C^{\rm s})<h_{\rm top}(\Sigma)$. 
 Let
 $\mu\in \mathcal M(\Sigma,\sigma)$ be ergodic with $h_\mu(\sigma)>h(\mathcal C^{\rm p}\cup\mathcal C^{\rm s})$ and let $\varphi\colon\Sigma\to\mathbb R$ be continuous. 
Let $\varepsilon\in(0,h_\mu(\sigma))$ and $M\geq1$.
The proof of proposition~\ref{nov***} breaks into two steps. In step~1 we construct the subset $\mathcal Q_k$ of $\mathcal G_k$ as in the statement. In step~2 we show that $\mathcal Q_k$ has all the required properties. 
\medskip

\noindent{\bf Step~1:  construction of $\mathcal Q_k$.}
  We use the next lemma in \cite{STY24} that can be proved combining Shannon-McMillan-Breiman's theorem and Birkhoff's ergodic theorem. Recall the notation for a prefix and suffix of a word in
   section~\ref{climen}.
   
\begin{lemma}[{\cite[lemma~2.6]{STY24}}]\label{extract-lem}
Let $\Sigma$ be a subshift, 
 $\mu\in \mathcal M(\Sigma,\sigma)$ be an ergodic measure with positive entropy and let 
 $\varphi\colon\Sigma\to\mathbb R$ be a continuous function. For any $\varepsilon\in(0,h_\mu(\sigma))$, there exists $N\in\mathbb N$ such that for each integer $n\geq N$ 
 there exists a subset $\mathcal P_{n}$ of $\mathcal L_{n}(\Sigma)$ such that:
 \begin{itemize}
\item[(a)] for all $w\in\mathcal P_{n}$,
$
e^{-n(h_\mu(\sigma)+\varepsilon)}\leq
\mu([w])\leq e^{-n(h_\mu(\sigma)-\varepsilon)}$;
\item[(b)] for all $w\in\mathcal P_{n}$ and all $x\in[w]$,
\[\left|\frac{S_{n}\varphi(x)}{n}- \int \varphi d\mu\right|<\varepsilon;
\]
\item[(c)]
$(1/2)e^{(h_\mu(\sigma)-\varepsilon)n}
\leq \#\mathcal P_{n}\leq e^{(h_\mu(\sigma)+\varepsilon)n}$;
\item[(d)] 
if $w\in\mathcal P_{n+1}$ then $w|_{n}^p\in\mathcal P_{n}$.
\end{itemize}
\end{lemma}
 For each sufficiently large 
 $n\in\mathbb N$,  let  
 $\mathcal P_{n}$
 denote the subset of $\mathcal L_{n}(\Sigma)$ given by lemma~\ref{extract-lem} with $\varepsilon$ replaced by $\varepsilon^3$. 
For each $w\in\mathcal P_{n}$ we fix a decomposition \[w=p(w)c(w)s(w),\ \
p(w)\in\mathcal C^{\rm p}, c(w)\in \mathcal G, s(w)\in\mathcal C^{\rm s}.\] 
 Set $\underline{h}(\mu)=(h_\mu(\sigma)-h(\mathcal C^{\rm p}\cup\mathcal C^{\rm s}))/2.$ 
From the proof of \cite[proposition~2.4]{STY24},
there exist $k$, $k_0\in\mathbb N$ and 
$p_0$, $s_0\in\mathbb N\cup\{0\}$
such that 
\begin{equation}\label{m-eq0}k=k_0-p_0-s_0,\ p_0+s_0\leq\frac{2\varepsilon^2 k_0}{\underline{h}(\mu)},\ M<\left(1-\frac{2\varepsilon^2}{\underline{h}(\mu)}\right)k_0\leq k \ \text{ and }\end{equation}
\begin{equation}\label{m-eq1}
\frac{\log\#\{c(w)\colon w\in\mathcal P_{k_0},\ |p(w)|=p_0,\ |s(w)|=s_0 \} }{k} 
> h_\mu(\sigma)-\varepsilon.
\end{equation}
We fix $u^*\in\mathcal L_{p_0}(\Sigma)$ such that \[\frac{\#\{c(w)\colon w\in\mathcal P_{k_0},\ p(w)=u^*,\ |s(w)|=s_0 \}}{\#\{c(w)\colon w\in\mathcal P_{k_0},\ |p(w)|=p_0,\ |s(w)|=s_0 \}}\geq\frac{1 }{\#\mathcal L_{p_0}(\Sigma) },\]
and define the set in the numerator to be $\mathcal Q_k(u^*).$
We fix $v^*\in\mathcal Q_{k}(u^*)$, and
for $i=1,\ldots,k$ let
 \[\begin{split}\mathcal Q_{k,i}^p&=\{v\in\mathcal Q_{k}(u^*)\colon
v|^s_{k-i+1}=v^*|^p_{k-i+1}\},\\
\mathcal Q_{k,i}^s&=\{v\in\mathcal Q_{k}(u^*)\colon
 v|^p_{k-i+1}=v^*|^s_{k-i+1}\},\end{split}\]
 and define a subset $\mathcal Q_k$ of $\mathcal G_k$ by
\[\mathcal Q_{k}=\mathcal Q_{k}(u^*)\setminus\bigcup_{i=1}^{\lfloor 2k/3+1\rfloor} (\mathcal Q_{k,i}^p\cup \mathcal Q_{k,i}^s).\]
Since $\{v^*\}=\mathcal Q_{k,1}^p=\mathcal Q_{k,1}^s$ we have
$v^*\notin\mathcal Q_{k}$. 
\medskip

\noindent{\bf Step~2:  verification of the required properties.} We have
\begin{equation}\label{seed-eq2}\begin{split} \#\mathcal Q_{k}&\geq \#\mathcal Q_{k}(u^*)- \sum_{i=1}^{\lfloor 2k/3+1\rfloor}  \#(\mathcal Q_{k,i}^p\cup  \mathcal Q_{k,i}^s).\end{split}\end{equation}
Let $2\leq i\leq \lfloor 2k/3+1\rfloor$ be an integer. 
For each $v\in\mathcal Q_k(u^*)$, there exists $w\in\mathcal P_{k_0}$ such that $s(w)\in\mathcal L_{s_0}(\Sigma)$ and
$w=u^*vs(w)$. By 
lemma~\ref{extract-lem}(d), the image of the injection 
 $v\in\mathcal Q_{k,i}^p\mapsto u^*(v|^p_{i-1})\in\mathcal L_{p_0+i-1}(\Sigma)$ is contained in $\mathcal P_{p_0+i-1}$. The second inequality of lemma~\ref{extract-lem}(c) yields
\begin{equation}\label{esti1}\#\mathcal Q_{k,i}^p\leq \#\mathcal P_{p_0+i-1}\leq e^{(h_\mu(\sigma)+\varepsilon)(p_0+i-1)}.\end{equation}
 If $v\in \mathcal Q_{k,i}^s$ then 
lemma~\ref{extract-lem}(d) gives
$u^*(v|^p_{k-i+1})=u^*(v^*|^s_{k-i+1})\in\mathcal P_{p_0+k-i+1}$
and $u^*v\in\mathcal P_{k_0-s_0}$.
This implies
\begin{equation}\label{esti2}
\begin{split}\#\mathcal Q_{k,i}^s\leq\frac{\mu([u^*
(v^*|^s_{k-i+1} )])}{\min\{\mu([u^*v])\colon v\in\mathcal Q^s_{k,i} \}}&\leq
\frac{e^{-(k-i+1)(h_\mu(\sigma)-\varepsilon)}}{e^{-k(h_\mu(\sigma)+\varepsilon) }}\\
&=e^{2\varepsilon k}e^{(i-1)(h_\mu(\sigma)-\varepsilon)}.\end{split}\end{equation}
Combining \eqref{esti1}, \eqref{esti2} and
$\#\mathcal Q_{k,1}^p=\#\mathcal Q_{k,1}^s=1$ yields
\[\sum_{i=1}^{\lfloor 2k/3+1\rfloor}\#(\mathcal Q_{k,i}^p\cup \mathcal Q_{k,i}^s)\leq (e^{2\varepsilon k}+1)\cdot\frac{e^{(2k/3)(h_\mu(\sigma)+\varepsilon)}-1}{e^{h_\mu(\sigma)+\varepsilon}-1}.\]
Substituting this inequality into \eqref{seed-eq2} and using \eqref{m-eq1} to estimate the first term in \eqref{seed-eq2}, we
 obtain the inequality in (a) for all sufficiently large $k$.

Let $\|\varphi\|=\sup_{x\in \Sigma}|\varphi(x)|$.
For any $w\in\mathcal P_{k_0}$ such that $c(w)\in\mathcal Q_k$,
take 
$y\in [w]$. 
Using the equality and the first inequality in \eqref{m-eq0}, 
we have
\[\left|\frac{S_{k}\varphi(\sigma^{p_0}y)}{k}-\frac{S_{k_0}\varphi(y)}{k}\right|\leq\frac{p_0+s_0}{k}\|\varphi\|\leq\frac{2\varepsilon^2 k_0 }{\underline{h}(\mu)k }\|\varphi\|,\]
and
by lemma~\ref{extract-lem}(b)  with $\varepsilon$ replaced by $\varepsilon^3$, 
\[\left|\frac{S_{k_0}\varphi(y)}{k}-\int \varphi d\mu\right|\leq\frac{k_0\varepsilon^3}{k}+\frac{k_0-k}{k}
\|\varphi\|.\]
Combining these two inequalities and  using the last inequality in \eqref{m-eq0}, we have
\[ \begin{split}\left|\frac{S_{k}\varphi(\sigma^{p_0}y)}{k}-\int \varphi d\mu\right|
< \frac{\varepsilon}{2}\end{split}\]
provided 
 $\varepsilon$ is sufficiently small depending on $\mu$ and $\varphi$.
Since $\varphi$ is uniformly continuous, if $k_0$ is sufficiently large then for all $x\in[c(w)]$ we have
\[\left|\frac{S_{k}\varphi(x)}{k}-\frac{S_{k}\varphi(\sigma^{p_0}y)}{k}\right|<\frac{\varepsilon}{2}.\]
Combining these two inequalities, we obtain 
\[ \left|\frac{S_{k}\varphi(x)}{k}- \int \varphi d\mu\right| < \varepsilon\ \text{ for all $x\in[c(w)]$,}\]
 which implies (b).
\medskip

Let $t\in\mathbb N\cup\{0\}$, 
 $w\in\mathcal L(\Sigma)$ satisfy $|w|\leq t<k/3$. Let
 $u,v\in\mathcal Q_k$ and 
 suppose $v^*$ is a subword of $uwv$. 
Since $|u|=|v|=|v^*|=k$, there would exist integers $i\in[2,k]$, $j\in[1,k-1]$ such that
$v^*=u|^s_{k-i+1}w(v|^p_j)$. We would have
$u\in\mathcal Q_{k,i}^p$ and $v\in\mathcal Q_{k,k-j+1}^s$, and moreover
 $k-i+1> k/3$ or $j> k/3$.
By the definition of $\mathcal Q_k$, the first inequality would yield $u\notin\mathcal Q_k$ and the second one would yield $v\notin\mathcal Q_k$. 
We have verified (c).
\end{proof}

\subsection{Construction of subsets of the recurrence set}\label{const-sec} 
Throughout this and the next subsections, 
let $\Sigma$ be a subshift  
that has a language decomposition $\mathcal L(\Sigma)=\mathcal C^{\rm p}\mathcal G\mathcal C^{\rm s}$ with $h(\mathcal C^{\rm p}\cup\mathcal C^{\rm s})<h_{\rm top}(\Sigma)$ and $\mathcal G$ has (W')-specification with gap size $t$.
Let $\mu\in\mathcal M(\Sigma,\sigma)$ be ergodic with $h_\mu(\sigma)>h(\mathcal C^{\rm p}\cup\mathcal C^{\rm s})$, let $\varphi\colon\Sigma\to\mathbb R$ be a continuous function,
and let $k>\max\{3t,1\}$ be a sufficiently large integer for which there exist $\mathcal Q_k\subset\mathcal G_k$ and $v^*\in\mathcal G_k\setminus\mathcal Q_k$ as in proposition~\ref{nov***}.

Let $a,b\in[0,\infty]$ satisfy $a\leq b$.
We construct a subset of the recurrence set $R_{\sigma,\xi}(a,b)$ by
modifying the seed set $F_k$ constructed in section~\ref{seed-sec}.
Let 
$(\ell_p)_{p=1}^\infty$,  $(\gamma_p)_{p=1}^\infty$ be sequences of positive reals with the properties in lemma~\ref{computational}.
By virtue of (b) and (c) in lemma~\ref{computational},
shifting the indices if necessary
 we may assume \begin{equation}\label{n-diff}\ell_{p+1}-\ell_p\geq k+t\ \text{ for all }p\in\mathbb N\ \text{ and }\end{equation}
\begin{equation}\label{gam}e^{\gamma_1\ell_1}\geq2k+t.\end{equation}

Let $x\in F_k$. Recall that $x=v^*(x) w^1(x)v^1(x)w^2(x)v^2(x)\cdots$
 with $v^i(x)\in \mathcal Q_{k}$,
 $w^i(x)\in\mathcal L(\Sigma)$ and $|w^i(x)|\leq t$
  for $i\in\mathbb N$. By lemma~\ref{computational}(e) and \eqref{gam}, we have
  \begin{equation}\label{gamII}e^{\gamma_{p}\ell_p}-e^{\gamma_{p-1}\ell_{p-1}}\geq
  e^{\gamma_{p-1}\ell_{p-1}}(e-1)\geq e^{\gamma_{1}\ell_{1}}(e-1)>k+t\ \text{ for all }p\geq2.\end{equation}
 We fix a strictly increasing sequence $(M_p(x))_{p=1}^\infty$ of positive integers such that 
\begin{equation}\label{relation3}
\begin{split}0&<e^{\gamma_{p}\ell_p}-|v^*w^1(x)
v^1(x)\cdots w^{M_p(x)-1}(x)v^{M_p(x)-1}(x)|\\
&\leq|w^{M_p(x)}(x) v^{M_p(x)}(x)|.\end{split}\end{equation}
By \eqref{gam} and \eqref{gamII}, $(M_p(x))_{p=1}^\infty$ is well-defined.
We construct a sequence $(y_q(x))_{q=0}^\infty$ of words in $\mathcal L(\Sigma)$ by induction. Start with $y_{0}(x)=v^*$. 
Suppose $q=0$, or else $q\geq1$ and
 $y_{q}(x )=v^*\cdots v^{q}(x)$ is a concatenation of $v^*$ and elements of $\mathcal Q_k$ by words of lengths not exceeding $t$. There are two cases.
 \medskip

\noindent {\bf Case 1:} $q\neq M_p(x)$ for all $p\in\mathbb N$.
  By the (W')-specification for $\mathcal G$ with gap size $t$, there is $w\in\mathcal L(\Sigma)$ such that 
 $y_{q}(x)wv^{q+1}(x)\in\mathcal L(\Sigma)$ and $|w|\leq t$. Define
\[y_{q+1}(x)=y_{q}(x)wv^{q+1}(x).\]

\medskip

\noindent {\bf Case 2:} $q=M_p(x)$ for some $p\in\mathbb N$.
Take a prefix $\theta^p(x)$ of $y_{q}(x)$
such that 
\begin{equation}\label{insert-ww}  k\leq|\theta^p(x)|,\ 
 \theta^p(x)|^s_k\in\mathcal Q_k\ \text{ and }\end{equation}
\begin{equation}\label{insert-w}0\leq |\theta^p(x)|- \ell_p<k+t.\end{equation}
Choose $\lambda^p(x)\in\mathcal Q_k$
such that for any $w\in\mathcal L(\Sigma)$
with $\theta^p(x)w\lambda^p(x)\in\mathcal L(\Sigma)$ and $|w|\leq t$, 
$\theta^p(x)w\lambda^p(x)$ is not a prefix of 
$y_{q}(x)$.
This choice is clearly feasible if $k$ is sufficiently large compared to $t$.
By the (W')-specification for $\mathcal G$ with gap size $t$, there are $w^{p,1},$ $w^{p,2}$, $w^{p,3}\in\mathcal L(\Sigma)$ such that 
$y_{q}(x)w^{p,1} \theta^p(x)w^{p,2}\lambda^p(x)w^{p,3}
v^{q+1}\in\mathcal L(\Sigma)$
and $|w^{p,i}|\leq t$ for $i=1,2,3$.
We set \[y_{q+1}(x)=y_{q}(x)w^{p,1} \theta^p(x)w^{p,2}\lambda^p(x)w^{p,3}
v^{q+1}.\]

Eligible choices of the connecting word $w$ in case~1 and that of $w^{p,1}$, $w^{p,2}$, $w^{p,3}$, $\lambda^p(x)$ in case~2 are not unique.
To avoid any ambiguity, we set a definite selection rule right from the start: introduce total orders in $\bigcup_{n=0}^t\mathcal L_n(\Sigma)$ and $\mathcal Q_k$, and choose from those eligible words the ones that are minimal with respect to these orders.
This completes the construction of the sequence $(y_q(x))_{q=0}^\infty$.
Let $y(x)$ denote the element of the singleton $\bigcap_{q=0}^\infty[y_q(x)]$. Clearly we have $y(x)\in\Sigma$. Let $\Lambda_{k}(a,b)$ denote the collection of these points, namely
\[\Lambda_k(a,b)=\{y(x)\in \Sigma\colon x\in F_k\}.\]

\begin{lemma}\label{tau-exp}For all $x\in F_k$, we have
\[\lim_{p\to\infty}\frac{\log\tau_{\xi_{|\theta^p(x)|}(y(x) )}(y(x) )}{\gamma_{p}\ell_p}=1.\]\end{lemma}
\begin{proof}
By \eqref{n-diff} and \eqref{insert-w}, the sequence $(\theta^p(x))_{p=1}^\infty$ 
is strictly monotone increasing and
\[\tau_{\xi_{|\theta^p(x)|}(y(x) )}(y(x) )=|y_{ M_p(x) }(x)w^{p,1}|.\]
From this and \eqref{relation3}, 
for all $p\geq2$ we have 
\[\begin{split}e^{\gamma_{p}\ell_p}\leq
\tau_{\xi_{|\theta^p(x)|}(y(x) )}(y(x) )\leq& 
e^{\gamma_{p}\ell_p}+k+t+\sum_{j=1}^{p-1} |w^{j,1}\theta^j(x)w^{j,2}
\lambda^j(x)w^{j,3}|\\
\leq& e^{\gamma_{p}\ell_p} +k+t+\sum_{j=1}^{p-1} |\theta^j(x)|+(p-1)(k+3t)\\
\leq& e^{\gamma_{p}\ell_p}+k+t+\sum_{j=1}^{p-1} \ell_j+(p-1)(2k+4t).
\end{split}\]
To deduce the last inequality we have used \eqref{insert-w}. The sum in the last line is bounded from above by 
 lemma~\ref{computational}(f) with $c=2k+4t$. Then, for any $\eta>0$ there exists $p_0\geq2$ such that for all $p\geq p_0$ we have
 \[e^{\gamma_{p}\ell_p}\leq
\tau_{\xi_{|\theta^p(x)|}(y(x) )}(y(x) )\leq(1+\eta)e^{\gamma_{p}\ell_p}+k+t,\]
which implies the desired equality.
 \end{proof}

\begin{lemma}\label{contained}
We have
$\Lambda_k(a,b)\subset R_{\sigma,\xi}(a,b)$.
\end{lemma}

\begin{proof}
Let $x\in F_k$. 
Combining \eqref{insert-w}, lemma~\ref{computational}(a)  and lemma~\ref{tau-exp}  we get
\begin{equation}\label{np-lem-eq1}\lim_{p\to\infty}
\frac{\log\tau_{\xi_{|\theta^{2p+1}(x)|}(y(x))}(y(x) )}{|\theta^{2p+1}(x)|}= a
\ \text{
and }\ \lim_{p\to\infty}\frac{\log
\tau_{\xi_{|\theta^{2p}(x)|}(y(x) )}(y(x) )}{|\theta^{2p}(x)|}=b.\end{equation} 
Similarly, combining \eqref{insert-w}, lemma~\ref{computational}(a), (d) and lemma~\ref{tau-exp} we get
\begin{equation}\label{equation-I}\limsup_{p\to\infty}\frac{\log\tau_{\xi_{|\theta^{p+1}(x)|}(y(x) )}(y(x) )}{|\theta^{p}(x)|}=\limsup_{p\to\infty}\frac{\gamma_{p+1}\ell_{p+1}}{\ell_p}\leq b.\end{equation}

For each $p\in\mathbb N$ define
\[N_1(p)=\{|\theta^p(x)|,\ldots,s_p(x)\}\ \text{ and } \ N_2(p)=\{s_p(x)+1,\ldots,|\theta^{p+1}(x)|-1\},\]
where
\[s_p(x)=\max\{i\geq|\theta^p(x)|\colon(\theta^p(x)w^{p,2 }\lambda^p(x))|^{\rm p}_i\text{ is a prefix of } y_{M_p(x) }(x)\}.\] 
By the choice of $\lambda^p(x)$, the number $s_p(x)$ is well-defined
and satisfies \[|\theta^p(x)|\leq s_p(x)< |\theta^p(x)w^{p,2}\lambda^p(x)|\leq|\theta^p(x)|+k+t.\]
By the inequality in \eqref{insert-ww} and $k>1$, $\theta^p(x)\geq2$ holds.
Notice that \[\mathbb N\setminus\{1,\ldots,|\theta^1(x)|-1\}=\bigsqcup_{p\in\mathbb N}(N_1(p)\bigsqcup N_2(p)).\]
Using proposition~\ref{nov***}(c) and the strict monotonicity of 
$(\theta^p(x))_{p=1}^\infty$, one can check by induction that 
\[\tau_{\xi_{n}(y(x) )}(y(x) )= \begin{cases}\tau_{\xi_{|\theta^1(x)|}(y(x)  )}(y(x) )\ &\text{ if }n\in\{1,\ldots, |\theta^1(x)|-1\},\\
\tau_{\xi_{|\theta^p(x)|}(y(x)  )}(y(x) ) \ &\text{ if }n\in N_1(p)\text{ for some }p\in\mathbb N,\\
\tau_{\xi_{|\theta^{p+1}(x)|}(y(x)  )}(y(x) ) \ &\text{ if }n\in N_2(p)\text{ for some }p\in\mathbb N.\end{cases}\]
Hence,
 if $n\in N_1(p)$ for some $p\in\mathbb N$ then 
\begin{equation}\label{equation-II}\frac{\log\tau_{\xi_{|\theta^{p}(x)|}(y(x) )}(y(x) )}{|\theta^{p}(x)|+k+t}\leq\frac{\log\tau_{\xi_{n}(y(x) )}(y(x) )}{n}\leq \frac{\log\tau_{\xi_{|\theta^{p}(x)|}(y(x) )}(y(x) )}{|\theta^{p}(x)|}.\end{equation}
Similarly, if $n\in N_2(p)$ for some $p\in\mathbb N$ then 
\begin{equation}\label{equation-III}\frac{\log\tau_{\xi_{|\theta^{p+1}(x)|}(y(x) )}(y(x) )}{|\theta^{p+1}(x)|}\leq\frac{\log\tau_{\xi_{n}(y(x) )}(y(x) )}{n}\leq \frac{\log\tau_{\xi_{|\theta^{p+1}(x)|}(y(x) )}(y(x) )}{|\theta^{p}(x)|}.\end{equation}
Combining \eqref{np-lem-eq1}, \eqref{equation-I}, \eqref{equation-II} and \eqref{equation-III} 
we obtain
\[ \liminf_{n\to\infty}\frac{\log\tau_{\xi_{n}(y(x) )}(y(x) )}{n}=a\ \text{ and }\ \limsup_{n\to\infty}\frac{\log\tau_{\xi_{n}(y(x) )}(y(x) )}{n}=b,\]
namely  $y(x)\in R_{\sigma,\xi}(a,b)$. Since $x\in F_k$ is arbitrary, we obtain $\Lambda_k(a,b)\subset R_{\sigma,\xi}(a,b)$ as required. \end{proof}

\subsection{Construction of Moran fractals}\label{moran-sec}
We continue the same setting as in section~\ref{const-sec}.

\begin{lemma}\label{upper-lem} 
There is $q(\varepsilon)\geq1$ such that for all  $x\in F_k$ and all $q\geq q(\varepsilon)$, we have \[|y_{q+1}(x)|\leq (1+\varepsilon)(k+t)(q+2).\] \end{lemma} \begin{proof} 
For $q\in\mathbb N$ and $x\in F_k$, define \[R(q,x)=\{p\in\mathbb N\colon |v^*w^1(x)v^1(x)\cdots w^{M_p(x)}(x)v^{M_p(x)}(x)|\leq (k+t)(q+2)\}.\] Clearly we have 
\begin{equation}\label{unif-R}\lim_{q\to\infty}\max\{R(q,x)\colon x\in F_k\}=\infty.\end{equation}

Let $x\in F_k$.
The definition of $y_{q+1}(x)$ in section~\ref{const-sec} implies \[|y_{q+1}(x)|\leq (k+t)(q+2)+\sum_{p\in R(q,x)}|w^{p,1}\theta^p(x)w^{p,2} \lambda^p(x)w^{p,3}|.\] Regarding the sum, \eqref{insert-w} gives \begin{equation}\label{eq-a-1}|w^{p,1}\theta^p(x)w^{p,2} \lambda^p(x)w^{p,3}|\leq \ell_p+k+t+k+3t.\end{equation} By virtue of \eqref{unif-R}, for all sufficiently large $q$ we have
\[\begin{split}\sum_{p\in R(q,x)} |w^{p,1}\theta^p(x)w^{p,2} \lambda^p(x)w^{p,3}|&\leq \varepsilon e^{\gamma_{\max R(q,x) } \ell_{\max R(q,x)}}\\ &\leq \varepsilon|v^*w^1(x)v^1(x)\cdots w^{M_{\max R(q,x)}(x)}(x)v^{M_{\max R(q,x)}(x)}(x)|\\ &\leq\varepsilon (k+t)(q+2).\end{split}\]
Indeed, the first inequality follows from \eqref{eq-a-1}  and lemma~\ref{computational}(f) with $c=2k+4t$. The second one follows from \eqref{relation3},  and the last one from the definition of $R(q,x)$. This yields the desired inequality.
\end{proof}
Let $v,w$ be words from $\{0,\ldots,m-1\}$. If $v$ is a prefix of $w$, then write $v\vartriangleleft w$. 
 For each $q\in\mathbb N$, we set \[\Delta_q(a,b)=\{V\in\mathcal L(\Sigma)\colon V=y_q(x)\text{ for some } x\in F_k\}.\]
 From the construction of the seed set $F_k$ in section~\ref{const-sec}, we have
\[\#\Delta_1(a,b)\geq \#\mathcal Q_{k},\]
and for all $q\geq1$ and all $V\in\Delta_q(a,b)$, 
  \[\#\{V'\in\Delta_{q+1}(a,b)\colon 
  \text{$V\vartriangleleft V'$} \}\geq\#\mathcal Q_{k}.\]
  Inductively on $q$, we extract elements from $\Delta_q(a,b)$ and construct a sequence $(\widetilde\Delta_q(a,b))_{q=1}^\infty$ of collections of words such that
  \begin{equation}\label{cardI}\#\widetilde\Delta_1(a,b)=\#\mathcal Q_{k},\end{equation} and 
   for all $q\geq1$ and all $V\in\widetilde\Delta_q(a,b)$, 
   \begin{equation}\label{cardII}\#\{V'\in\widetilde\Delta_{q+1}(a,b)\colon 
  \text{$V\vartriangleleft V'$} \}=\#\mathcal Q_{k}.\end{equation}
We set
\[\widetilde\Lambda_k(a,b)=\bigcap_{q=1}^\infty\bigcup_{V\in\widetilde\Delta_q(a,b)} [V].\]
Since $\Lambda_k(a,b)=\bigcap_{q=1}^\infty\bigcup_{V\in\Delta_q(a,b)} [V]$, we have
$\widetilde\Lambda_k(a,b)\subset\Lambda_k(a,b)$.

For each $y=(y_i)_{i\in\mathbb N}\in\widetilde\Lambda_k(a,b)$ and $n\in\mathbb N$, 
let $q_n(y)\geq0$ denote the minimal integer 
such that $y_1\cdots y_{n}$ is a prefix of some element of $\widetilde\Delta_{q_n(y)+1}(a,b)$.

\begin{lemma}\label{q-lem} For all sufficiently large $n$, we have \[\inf\{q_n(y)\colon y\in \widetilde\Lambda_k(a,b)\}\geq\frac{1}{1+2\varepsilon}\frac{n}{k+t}.\]\end{lemma}
\begin{proof} Clearly we have
\begin{equation}\label{unif-q}\lim_{n\to\infty}\inf\{q_n(y)\colon y\in \widetilde\Lambda_k(a,b)\}=\infty.\end{equation}

By virtue of \eqref{unif-q} and lemma~\ref{upper-lem}, for all sufficiently large $n\geq1$
we have $n\leq|y_{q_n(y(x))+1}(x )|\leq(1+\varepsilon)(k+t)(q_n(y(x))+2)$ for all $x\in F_k$ with $y(x)\in\tilde\Lambda_k(a,b)$, and so \[q_n(y(x))\geq\frac{1}{1+\varepsilon}\frac{n}{k+t}-2\geq\frac{1}{1+2\varepsilon}\frac{n}{k+t}.\] The last inequality holds for all sufficiently large $n$.
\end{proof}

Notice that the map $x\in F_k\mapsto y(x)\in\Lambda_k(a,b)$ is injective.
Let
$f_{a,b}\colon \Lambda_k(a,b)\to  F_k$
denote the inverse of this map.

\begin{lemma}\label{counting}
For all $y\in\widetilde\Lambda_k(a,b)$ and 
all $n\in\mathbb N$, we have 
\[
\#\left\{
\begin{tabular}{l}
\!\!\!$V\in\widetilde\Delta_{q_n(y)+1}(a,b)\colon $\!\!\! $y_1\cdots y_{n}\vartriangleleft V$, and $\exists z\in\widetilde\Lambda_k(a,b)$ such that \!\!\!\\
\quad  $q_n(y)\neq M_p(f_{a,b}(z))$ for all $p\in\mathbb N$ and $V=y_{q_n(y)+1}(f_{a,b}(z))$ \end{tabular}
\right\}\leq m^{k+t+1}, 
\]
and
\[
\#\left\{
\begin{tabular}{l}
\!\!\!$V\in\widetilde\Delta_{q_n(y)+1}(a,b)\colon $\!\!\! $y_1\cdots y_{n}\vartriangleleft V$, and $\exists z\in\widetilde\Lambda_k(a,b)$ such that \!\!\!\\
\quad $q_n(y)=M_p(f_{a,b}(z))$ for some $p\in\mathbb N$ and $V=y_{q_n(y)+1}(f_{a,b}(z))$ \end{tabular}
\right\}\leq m^{2k+3t+3}.
\]
\end{lemma}
\begin{proof}
Let $V$ belong to the first set in the statement of lemma~\ref{counting}. By the minimality of $q_n(y)$, there exists $z\in\widetilde\Lambda_k(a,b)$ such that
$V=y_{q_n(y)}(x )wv^{q+1}(x)$ 
where $x=f_{a,b}(z)$, $y_{q_n(y)}(x )\vartriangleleft y_1\cdots y_n$,
$v^{q+1}(x)\in\mathcal Q_k$, $w\in\mathcal L(\Sigma)$ and $|w|\leq t$. The total number of words of this form is bounded from above by $\#\bigcup_{n=0}^t\mathcal L_n(\Sigma)\cdot\#\mathcal Q_k\leq m^{k+t+1}$.

Let $V$ belong to the second set in the statement of lemma~\ref{counting}.
By the minimality of $q_n(y)$, there exists $z\in\widetilde\Lambda_k(a,b)$ such that \[V=y_{q_n(y)}(x )w^{p,1}\theta^p(x )w^{p,2 }\lambda^p(x )w^{p,3 }v^{q_n(y)+1}(x ),\]
where $x=f_{a,b}(z)$, $y_{q_n(y)}(x)\vartriangleleft y_1\cdots y_n$,
 $\theta^p(x)$, $v^{q_n(y)+1}(x )$, $\lambda^p(x)\in\mathcal Q_k$,
 $w^{p,i}\in\mathcal L(\Sigma)$ and 
 $|w^{p,i}|\leq t$ for $i=1,2,3$. Combining lemma~\ref{computational}(c), \eqref{relation3} and
\eqref{insert-w} we get \[|\theta^p(x )|< \ell_{q_n(y)}+k+t\leq e^{\gamma_{q_n(y)} \ell_{q_n(y)}}\leq |y_{q_n(y)}(x )|\] provided $n$ is sufficiently large independent of $y$. In particular,
$\theta^p(x )$ is a prefix of $y_1\cdots y_n$. 
The total number of words of this form is bounded from above by $(\#\bigcup_{n=0}^t\mathcal L_n(\Sigma))^3(\#\mathcal Q_k)^2\leq m^{2k+3t+3}$ as required. \end{proof}

\subsection{Proof of theorem~\ref{mainthm}}\label{end-sec}
Let $\Sigma$ be a subshift  
that has a language decomposition $\mathcal L(\Sigma)=\mathcal C^{\rm p}\mathcal G\mathcal C^{\rm s}$ such that $h(\mathcal C^{\rm p}\cup\mathcal C^{\rm s})<h_{\rm top}(\Sigma)$ and $\mathcal G$ has (W')-specification with gap size $t$. 
Let $H\in(h(\mathcal C^{\rm p}\cup\mathcal C^{\rm s}), h_{\rm top}(\Sigma))$. By the variational principle and Jacobs' theorem \cite{Wal82}, there is an ergodic measure $\mu\in\mathcal M(\Sigma,\sigma)$ such that
\begin{equation}\label{ent}h_\mu(\sigma)>H.\end{equation} Let $\varphi\colon\Sigma\to\mathbb R$ be a constant function. 
Let $k>\max\{3t,1\}$ be a sufficiently large integer for which there exist $\mathcal Q_k\subset\mathcal G_k$ and $v^*\in\mathcal G_k\setminus\mathcal Q_k$ as in proposition~\ref{nov***}.
Let $a,b\in[0,\infty]$ satisfy $a\leq b$.

 For each $V\in \widetilde\Delta_q(a,b)$, fix a point 
 $x(V)\in [V]\cap\widetilde\Lambda_k(a,b)$. Define a Borel probability measure $\nu_q$ on $\Sigma$ by 
\[\nu_q=\frac{1}{\#\widetilde\Delta_q(a,b) }\sum_{V\in \widetilde\Delta_q(a,b) }\delta_{x(V)},\] 
where $\delta_{x(V)}$ denotes the unit point mass at $x(V)$.
Pick an accumulation point of the sequence $(\nu_q)_{q=1}^\infty$ in the weak* topology and denote it by $\nu$.
Since $\widetilde\Lambda_k(a,b)$ is a closed subset of $\Sigma$ and $\nu_q(\widetilde\Lambda_k(a,b))=1$,
we have $\nu(\widetilde\Lambda_k(a,b))=1$. In view of lemma~\ref{counting}, put $C_{k,t}=m^{k+t+1}+m^{2k+3t+3}$.

Let $y\in\widetilde\Lambda_k(a,b)$ and $n\in\mathbb N$.
Recall that $q_n(y)\geq0$ is the minimal integer 
such that $y_1\cdots y_{n}$ is a prefix of some element of $\widetilde\Delta_{q_n(y)+1}(a,b)$. 
By \eqref{cardII}, lemma~\ref{q-lem} and lemma~\ref{counting}, 
for all sufficiently large $n$ and all $q\geq q_n(y)$ we have
\[\begin{split}\#\{V\in\widetilde\Delta_{q+1}(a,b)\colon x(V)\in [y_1\cdots y_{n}] \}
=& \#\{V\in\widetilde\Delta_{q+1}(a,b)\colon y_1\cdots y_{n}\vartriangleleft V \}\\
=&(\#\mathcal Q_k)^{q-q_n(y)}\\
&\times\#\{V\in\widetilde\Delta_{q_n(y)+1}(a,b)\colon y_1\cdots y_{n}\vartriangleleft V \}\\
\leq& C_{k,t}(\#\mathcal Q_k)^{q-\frac{1}{1+2\varepsilon}\frac{n}{k+t} }.\end{split}\] 
Since
 $\#\widetilde\Delta_{q+1}(a,b)=(\#\mathcal Q_k)^{q+1}$ by \eqref{cardI} and \eqref{cardII}, we have 
\[\begin{split}\nu_{q+1}([y_1\cdots y_{n}] )&=\frac{\#\{V\in\widetilde\Delta_{q+1}(a,b)\colon x(V)\in [y_1\cdots y_{n}]\}}{\#\widetilde\Delta_{q+1}(a,b)}\\
&\leq C_{k,t}(\#\mathcal Q_{k})^{-\frac{1}{1+2\varepsilon}\frac{n}{k+t}}=C_{k,t}(e^{-n})^{\frac{\log \#\mathcal Q_{k} }{(1+2\varepsilon)(k+t) } },\end{split}\]
and thus
\[\nu([y_1\cdots y_{n}])\leq\limsup_{q\to\infty}\nu_{q+1}([y_1\cdots y_{n}])\leq C_{k,t}(e^{-n})^{\frac{\log \#\mathcal Q_{k}}{(1+2\varepsilon)(k+t) } }.\]
Since $y$ is an arbitrary point of $\widetilde\Lambda_k(a,b)$ and 
 $[y_1\cdots y_{n}]$ is the ball of radius $e^{-n}$ about  $y\in\widetilde\Lambda_k(a,b)$,
the mass distribution principle 
yields 
\[\dim_{\rm H}\widetilde\Lambda_{k}(a,b)\geq
   \frac{\log \#\mathcal Q_{k} }{(1+2\varepsilon)(k+t) }.\]
By lemma~\ref{contained},
 proposition~\ref{nov***}(a) and \eqref{ent},
\[\dim_{\rm H}R_{\sigma,\xi}(a,b)\geq\dim_{\rm H}\widetilde\Lambda_{k}(a,b)\geq
\frac{k(H-\varepsilon)}{(1+2\varepsilon)(k+t)}.\]
   Increasing $k$ to $\infty$, and then decreasing $\varepsilon$ to $0$ and then increasing $H$ to $h_{\rm top}(\Sigma)$, we obtain the desired inequality in \eqref{lower}.
The proof of theorem~\ref{mainthm} is complete. \qed

\section{Dimension of recurrence sets for interval maps}
This section is devoted to the analysis of recurrence sets for interval maps in theorem~\ref{monotone-thm}.
In section~\ref{Markov-sec} we introduce a Markov diagram, a directed graph associated with a given interval map. Using a Markov diagram, we show that the coding space of any interval map in theorem~\ref{monotone-thm} satisfies the assumption in theorem~\ref{mainthm}. In section~\ref{lastlast}, we exploit the Markov diagram and the Moran fractals constructed in section~\ref{moran-sec} to  complete the proof of theorem~\ref{monotone-thm}. 
\subsection{Markov diagram}\label{Markov-sec}
Let $T\colon[0,1]\to[0,1]$ be a transitive piecewise monotonic map 
with positive topological entropy.
Let \[X_T=\bigcap_{n=0}^\infty T^{-n}\left(\bigcup_{j=0}^{m-1} {\rm int}(I_j) \right),\] and
define $\pi\colon X_T\to \Sigma_m$ by
$x\in \bigcap_{n=0}^\infty T^{-n}({\rm int}(I_{(\pi(x))_{n+1}}) ) $.
The closure of $\pi(X_T)$ in $\Sigma_m$ is a subshift, called  {\it the coding space} of $T$ and denoted by $\Sigma_T$. Notice that $\pi\circ T|_{X_T}=\sigma\circ\pi$.
It is easy to see that if $T$ is transitive then $\pi$ is injective. If moreover $T$  has positive topological entropy, then $h_{\rm top}(\Sigma_T)>0$.

Let $C$ be a non-empty closed subset of $\Sigma_T$ that is contained in one of the $1$-cylinders. 
We say a non-empty closed subset $D$ of $\Sigma_T$ is {\it a successor} of $C$ if there exists $j\in\{0,\ldots,m-1\}$ such that
$D=[j]\cap\sigma C$. If $D$ is a successor of $C$, we write $C\to D$.
We set
$\mathcal{D}_0=\{[0],\ldots,[m-1]\}$,
and define $\mathcal D_n$, $n=1,2,\ldots$ 
by the recursion formula
\[\mathcal{D}_{n}=\mathcal{D}_{n-1}\cup \{D\colon D\text{ is a successor of some }C\in\mathcal{D}_{n-1}\}.\]
We set
\[\mathcal{D}=\bigcup_{n= 0}^\infty\mathcal{D}_n.\]
The directed graph $(\mathcal D,\to)$ is called {\it a Markov diagram} associated with $\Sigma_T$.

For a set $\mathcal C\subset\mathcal D$,
 let $\Sigma_\mathcal C$ denote the set of 
 paths 
 $C_1\to C_2\to \cdots$
 that contains infinitely many edges and all whose vertices are contained in $\mathcal C$.
We define a map
$\Psi\colon \Sigma_{\mathcal{D}}\to\{0,\ldots,m-1\}^{\mathbb N}$ as follows.
For each vertex $C_1$ or each path $C_1\to \cdots \to C_{n}$ in the diagram $(\mathcal D,\to)$, define
\[(C_1\cdots C_{n})=x_1\cdots x_{n}\in\{0,\ldots,m-1\}^n,\]
where $C_i\subset [x_i]$ for $i=1,\ldots,n$. 
For $(C_n)_{n=1}^\infty\in \{0,\ldots,m-1\}^\mathbb N$
define
\[ \Psi((C_n)_{n=1}^\infty)\in
\bigcap_{n=1}^\infty[(C_1\cdots C_{n})].\]
Note that $\Sigma_\mathcal D$ is a Markov shift over the countably infinite alphabet $\mathcal D$, $\Psi(\Sigma_\mathcal D)=\Sigma_T$, and 
$\Psi$ semiconjugates
the left shift on $\Sigma_\mathcal D$ to that on
$\Sigma_T$.

\begin{lemma}[{\cite[pp.224, 226]{HR98}}]\label{irred-lem}There exist a subset
$\mathcal{C}$ of $\mathcal{D}$ and a finite subset
$\mathcal{F}$ of $\mathcal{C}$
such that:
\begin{itemize}
\item[(a)]  for any pair $(C,C')$ of vertices in $\mathcal C$ there is a path that leads from $C$ to $C'$;
\item[(b)]
$C\in\mathcal{C}$ and $C\rightarrow D$
imply $D\in\mathcal{C}$;
\item[(c)] 
$\Psi(\{(C_n)_{n=1}^\infty\in\Sigma_{\mathcal{C}}\colon C_1\in\mathcal{F}\})
=\Psi(\Sigma_{\mathcal{C}})=\Sigma_T$.
\end{itemize}
\end{lemma}

\begin{remark}
In \cite[pp.224, 226]{HR98}, a stronger statement than lemma~\ref{irred-lem} was used that follows from \cite[theorems~10, 11]{H}.
\end{remark}

Let $\mathcal C\subset\mathcal D$ and let $\mathcal F\subset\mathcal C$ be a finite set for which the conclusion of lemma~\ref{irred-lem} holds.
For each integer $N\geq0$
we define subsets
 $\mathcal C^{{\rm p},N}$, $\mathcal G^N$, $\mathcal C^{{\rm s},N}$ of $\mathcal L(\Sigma_T)$ by
\begin{equation}\label{language-decomp}\begin{split}\mathcal C^{{\rm p},N}&=\{\emptyset\},\\
\mathcal G^N&=\{\emptyset\}\cup\bigcup_{n=1}^\infty
\{(C_1\cdots C_{n})\colon C_1\to\cdots\to C_{n},\ C_1,C_{n}\in
\mathcal{D}_N\cap\mathcal{C}\},\\
\mathcal C^{{\rm s},N}&=\{\emptyset\}\cup\bigcup_{n=1}^\infty \left\{\begin{tabular}{l}\vspace{1mm}
$\!\!\displaystyle{(C_1\cdots C_{n})\colon C_1\to\cdots \to C_{n},
\!\!}$\\
$\!\!\displaystyle{ C_1\in
\mathcal{D}_{N+1},\
C_i\in\mathcal{C}\setminus\mathcal{D}_N\text{ for }i=1,\ldots,n}$\end{tabular}
\!\!\right\}.\end{split}\end{equation}

\begin{lemma}\label{Mar-lem}The following statements hold:
\begin{itemize}
\item[(a)] for every $N\geq0$, $\mathcal G^N$ has (W')-specification with gap size $t_N$;

\item[(b)] if $\mathcal F\subset\mathcal D_N$, then $\mathcal L(\Sigma_T)=\mathcal C^{{\rm p},N}\mathcal G^N\mathcal C^{{\rm s},N}$;

\item[(c)] $\lim_{N\to\infty}h(\mathcal C^{{\rm s},N})=0$.
\end{itemize}
\end{lemma}
\begin{proof}Let
$t_N\geq1$ denote the minimal integer such that for any pair $(C,C')$ of vertices in $\mathcal{D}_{N}\cap \mathcal{C}$ there is a path that leads from $C$ to $C'$ and does not contain more than $t_N$ edges.
By lemma~\ref{irred-lem}(a) and the finiteness of $\mathcal{D}_{N}$,
 $t_N$ is well-defined. 
 
 Let $u,v\in\mathcal G^N$.
There exist two paths
$C_1\to\cdots \to C_{|u|}$ and $D_1\to\cdots\to D_{|v|}$
such that $(C_1\cdots C_{|u|})=u$, $C_{|u|}\in\mathcal{D}_{N}\cap \mathcal{C}$
and $(D_1\cdots D_{|v|})=v$,  $D_1\in\mathcal{D}_{N}\cap\mathcal{C}$.
Hence, there is a path
$C_{|u|}\to\cdots\to D_1$ 
not containing more than $t_N$ edges, which implies (a). 
The proof of \cite[proposition~4.1]{STY24} shows (b) and (c). \end{proof}

Lemma~\ref{Mar-lem} implies the next proposition.
\begin{prop}
\label{CT-dec}
Let $T\colon[0,1]\to[0,1]$ be a transitive piecewise monotonic map
with positive topological entropy.
For any $\varepsilon>0$ there is a decomposition
$\mathcal L(\Sigma_T)=\mathcal C^{\rm p}\mathcal G\mathcal C^{\rm s}$ such that $\mathcal G$ has (W')-specification and $h(\mathcal C^{\rm p}\cup\mathcal C^{\rm s})<\varepsilon$.
\end{prop}

\subsection{Proof of theorem~\ref{monotone-thm}}\label{lastlast}

Let $T\colon[0,1]\to[0,1]$ be a transitive piecewise expanding map of class $C^2$.
By \cite{LY73} and \cite[corollary~4]{L81}, there exists a $T$-invariant ergodic Borel probability measure $\nu$ that is  absolutely continuous with respect to the Lebesgue measure on $[0,1]$. Put $\chi_\nu(T)=\int\log|T'|d\nu$.
Since $T$ is expanding, $\chi_\nu(T)>0$.
By
\cite{R64} and \cite[theorem~3]{L81}, Rohlin's formula holds:
\begin{equation}\label{rok}h_\nu(T)=
\chi_\nu(T).\end{equation}
Clearly we have $\nu(X_T)=1$.
The measure $\mu=\nu\circ\pi^{-1}$ belongs to
$\mathcal M(\Sigma_T,\sigma)$ and satisfies $h_\mu(\sigma)=h_\nu(T)$. 
Let $\varepsilon\in(0,h_\mu(\sigma))$.
By lemma~\ref{Mar-lem}, there exist $N\in\mathbb N$ and a decomposition
$\mathcal L(\Sigma_T)=\mathcal C^{{\rm p},N}\mathcal G^N\mathcal C^{{\rm s},N}$ such that $\mathcal G^N$ has (W')-specification and $h(\mathcal C^{{\rm p},N}\cup\mathcal C^{{\rm s},N})<\varepsilon$. 

Since $T$ is expanding of class $C^2$, the function $x\in \pi(X_T)\mapsto \log|T'(\pi^{-1}(x))|\in\mathbb R$ is uniformly continuous. 
Let $\varphi$
denote the continuous extension of this function to $\Sigma_T$.
Let $k>\max\{3t,1\}$ be a sufficiently large integer for which there exist $\mathcal Q_k\subset\mathcal G_k$ and $v^*\in\mathcal G_k\setminus\mathcal Q_k$ as in proposition~\ref{nov***}  applied to
the above decomposition of $\mathcal L(\Sigma_T)$. 
Let 
${\rm diam}(\cdot)$ denote the Euclidean diameter of a subset of $[0,1]$, and
set \[K_N=\inf\{{\rm diam}(\pi^{-1}(D))\colon D\in\mathcal D_N\}.\] Since $\mathcal D_N$ is a finite set, $K_N$ is positive.

Let $a,b\in[0,\infty]$ satisfy $a\leq b$.
We have $\pi^{-1}(\widetilde\Lambda_k(a,b))\subset R_{T,\xi_T}(a,b)$.
To estimate the Hausdorff dimension of
$\pi^{-1}(\widetilde\Lambda_k(a,b))$ from below,
for each $q\in\mathbb N$ we put \[\Gamma_q(a,b)=\{\overline{\pi^{-1}([V])}\colon V\in\widetilde\Delta_q(a,b)\}.\]
Elements of $\Gamma_q(a,b)$
are non-degenerate closed subintervals of $[0,1]$ that intersect $\pi^{-1}(\widetilde\Lambda_k(a,b))$ and have mutually disjoint interiors.

\begin{lemma}\label{diam-lem}
For 
all $q\in\mathbb N$ and all $B\in\Gamma_{q}(a,b)$, we have
\[{\rm diam}(B)\geq K_N\exp(-(\chi_\nu(T)+2\varepsilon)kq).\]\end{lemma}
\begin{proof}
From the definition of $\mathcal G^N$ in \eqref{language-decomp}, for any $B\in\Gamma_{q}(a,b)$ 
there exist $n\in\mathbb N$,  $V\in\widetilde\Delta_q(a,b)$, $D\in\mathcal D_N$ such that
$\overline{\pi(B\cap X_T)}=[V]$, $\sigma^{|V|}[V]\supset D$, and $T^{|V|}|_B$ is a $C^1$ diffeomorphism onto its image and 
$T^{|V|}(B)\supset\overline{\pi^{-1}(D)}$. 
By the mean value theorem, there exists $x\in B$ such that \begin{equation}\label{comb1}|(T^{|V|})'x|{\rm diam}(B)={\rm diam}(\pi^{-1}(D)).\end{equation}
By proposition~\ref{nov***}(b) we have 
\begin{equation}\label{comb2}|(T^{|V|})'x|\leq\exp((\chi_\nu(T)+\varepsilon)kq)(\sup|T'|)^{tq}\leq 
\exp((\chi_\nu(T)+2\varepsilon)kq),\end{equation}where the last inequality holds provided $k$ is sufficiently large. Combining \eqref{comb1} and \eqref{comb2} yields the desired inequality. \end{proof}

Similarly to section~\ref{end-sec},
 for each $B\in \Gamma_q(a,b)$ fix a point
 $x(B)$ in ${\rm int}(B)\cap\pi^{-1}(\widetilde\Lambda_k(a,b))$, and define a Borel probability measure $\nu_q$ on $[0,1]$ by 
\[\nu_q=\frac{1}{\#\Gamma_q(a,b) }\sum_{B\in \Gamma_q(a,b) }\delta_{x(B)},\] 
where $\delta_{x(B)}$ denotes the unit point mass at $x(B)$. Pick an accumulation point of the sequence $(\nu_q)_{q=1}^\infty$ in the weak* topology and denote it by $\nu$.
Since $\widetilde\Lambda_k(a,b)$ is a closed subset of $\Sigma_T$ and $\pi$ is continuous, $\pi^{-1}(\widetilde\Lambda_k(a,b))$ is a closed subset of $[0,1]$.
Since $\nu_q(\pi^{-1}(\widetilde\Lambda_k(a,b)) )=1$ for all $q\in\mathbb N$,
we have $\nu(\pi^{-1}(\widetilde\Lambda_k(a,b)))=1$.

\begin{lemma}\label{final-e}For all $q\in\mathbb N$ and all $B\in\Gamma_q(a,b)$, we have 
\[\nu(B)=\frac{1 }{\#\Gamma_{q}(a,b)}.\]\end{lemma}
\begin{proof}Let $q\in\mathbb N$.
By construction, for all $s>q$ and all $B\in\Gamma_q(a,b)$ we have 
\[\nu_{s}(B)=\frac{\#\{C\in\Gamma_{s}(a,b)\colon C\subset B\}}{\#\Gamma_{s}(a,b)}=\frac{1 }{\#\Gamma_{q}(a,b)}.\]
Since $\nu(\partial B)=0$,
letting $s\to\infty$ yields the desired equality.\end{proof}

Let $x\in\pi^{-1}(\widetilde\Lambda_k(a,b))\cap(0,1)$.
For each sufficiently small $r>0$, take $q\in\mathbb N$ with
\[K_N\exp(-(\chi_\nu(T)+2\varepsilon)k(q+1))<r\leq K_N\exp(-(\chi_\nu(T)+2\varepsilon)kq).\]
By lemma~\ref{diam-lem}, the interval $[x-r,x+r]$ in $[0,1]$ intersects at most three elements of $\Gamma_q(a,b)$.
Lemma~\ref{final-e} gives
\[\nu([x-r,x+r])\leq \frac{3}{\#\Gamma_q(a,b)}.\]
Proposition~\ref{nov***}(a) gives $\#\Gamma_q(a,b)=\#\widetilde\Delta_q(a,b)=(\#\mathcal Q_k)^q\geq \exp((h_\nu(T)-\varepsilon)kq)$.
Plugging this inequality into the previous one, taking logarithms, dividing by $\log r$ and rearranging the result yields
\[\liminf_{r\to0}\frac{\log\nu([x-r,x+r])}{\log r}\geq\frac{h_{\nu}(T)-\varepsilon}{\chi_\nu(T)+2\varepsilon}.\]
Since $x$ is an arbitrary point of $\pi^{-1}(\widetilde\Lambda_k(a,b))\cap(0,1)$, 
the mass distribution principle shows
\[\dim_{\rm H}R_{T,\xi_T}(a,b)\geq\dim_{\rm H}\pi^{-1}(\widetilde\Lambda_k(a,b))\geq\frac{h_{\nu}(T)-\varepsilon}{\chi_\nu(T)+2\varepsilon}.\]
Decreasing $\varepsilon$ to $0$ and then using \eqref{rok}, we obtain the desired equality in theorem~\ref{monotone-thm}.  
The proof of theorem~\ref{monotone-thm} is complete. \qed

\appendix
\def\thesection{\Alph{section}}

\section{}

\subsection{Proof of lemma~\ref{computational}}\label{compute-s}
 First we define 
 $(\ell_{p})_{p=1}^\infty$ or $(\ell_{p})_{p=0}^\infty$ inductively, and then define
 $(\gamma_p)_{p=1}^\infty$ depending on $(a,b)$ as follows:

 \medskip
\noindent{\bf Case I:}  $0<a\leq b<\infty$. 
 \[\ell_1=b^{-2}+1,\  
  \ell_{p+1}=\begin{cases}(b/a)(
  \ell_p+\sqrt{\ell_p})&\text{if $p$ is even},\\
\ell_p+\sqrt{\ell_p}&\text{if $p$ is odd,}
 \end{cases}\]
 \[\gamma_p=\begin{cases}
 b& \text{ if $p$ is even,}\\
 a& \text{ if $p$ is odd};
 \end{cases}\]

\noindent {\bf Case II:} $0=a<b<\infty$.
  \[\ell_0=\ell_1=b^{-2}+1,\ 
  \ell_{p+1}=\begin{cases}b
  \sqrt{\ell_p}(\ell_p+\sqrt{\ell_p})&\text{if $p$ is even},\\
 \ell_p+\sqrt{\ell_p}&\text{if $p$ is odd,}
 \end{cases}\]
 
 \[\gamma_{p}=\begin{cases}
 b&\text{ if $p$ is even,}\\
 1/\sqrt{\ell_{p-1}}&\text{ if $p$ is odd};
 \end{cases}\]

\noindent{\bf Case III:}   $0<a<b=\infty$. 
 \[\ell_1=a+1,\ \ell_{p+1}=\begin{cases}(\ell_p/a)(
  \ell_p+\sqrt{\ell_p})&\text{if $p$ is even},\\
 \ell_p+\sqrt{\ell_p}&\text{if $p$ is odd,}
 \end{cases}\]
 \[
 \gamma_p=\begin{cases}
 \ell_p&\text{ if $p$ is even,}\\
 a&\text{ if $p$ is odd;}
 \end{cases}\]

\noindent{\bf Case IV:} 
 $a=b=0$.
\[\ell_1=1,\ 
  \ell_{p+1}=
 \ell_p+\sqrt{\ell_p}\ \text{ and }\ \gamma_{p}=1/\sqrt{\ell_p}\ \text{ for $p\in\mathbb N$};
 \]

 \noindent{\bf Case V:} 
 $a=b=\infty$. 
 \[\ell_1=1,\ \ell_{p+1}=
 \ell_p+\sqrt{\ell_p}\ \text{ and } \ \gamma_{p}=\ell_p\ \text{ for $p\in\mathbb N$;}
 \]

\noindent{\bf Case VI:} 
$0=a<b=\infty$.
 \[\ell_0=\ell_1=1,\ \ell_{p+1}=\begin{cases}\ell_p\sqrt{\ell_p} (
  \ell_p+\sqrt{\ell_p})&\text{if $p$ is even},\\
 \ell_p+\sqrt{\ell_p}&\text{if $p$ is odd,}
 \end{cases}\]
 \[\gamma_{p}=\begin{cases}
 \ell_p&\text{ if $p$ is even,}\\
 1/\sqrt{\ell_{p-1}}&\text{ if $p$ is odd.}
 \end{cases}\]
It is immediate to check (a) (b) (c) (d) 
 in all these cases.

 To prove (e)
we treat the six cases one by one.
\medskip

\noindent{\bf Case I:}  $0<a\leq b<\infty$. 
By definition, for all $p\in\mathbb N$ we have
\begin{equation}\label{p-ineq}\ell_p\geq \ell_1>b^{-2}.\end{equation}
If $p$ is even then \eqref{p-ineq} gives
$\gamma_p \ell_p+1=b \ell_p+1<b(\ell_p+\sqrt{\ell_p})=\gamma_{p+1}\ell_{p+1}$.
If $p$ is odd then 
 \eqref{p-ineq} gives $\gamma_p \ell_p+1=a \ell_p+1<b(\ell_p+\sqrt{\ell_p})=\gamma_{p+1}\ell_{p+1}.$

\medskip

\noindent {\bf Case II:} $0=a<b<\infty$.
By definition, for all $p\in\mathbb N$ we have
\begin{equation}\label{p-ineq'}\ell_p\geq \ell_1>b^{-2}.\end{equation}
   If $p$ is even then $\ell_p>b^{-2}$ in \eqref{p-ineq'} gives 
$\gamma_p\ell_p+1=b \ell_p+1<b(\ell_p+\sqrt{\ell_p})=\gamma_{p+1}\ell_{p+1}$.
    If $p$ is odd then 
     \eqref{p-ineq'} gives 
  $\gamma_p\ell_p+1=(1/\sqrt{\ell_{p-1}}) \ell_p+1\leq b \ell_p+1<b(\ell_p+\sqrt{\ell_p})=\gamma_{p+1}\ell_{p+1}$.

\medskip

\noindent{\bf Case III:}   $0<a<b=\infty$. 
By definition, for all $p\in\mathbb N$ we have
\begin{equation}\label{p-ineq''}\ell_p\geq \ell_1>\max\{1,a\}.\end{equation}
If $p$ is even then $\ell_p>1$ in \eqref{p-ineq''} gives 
$\gamma_p \ell_p+1=\ell_p^2+1<\ell_p(\ell_p+\sqrt{\ell_p})=\gamma_{p+1}\ell_{p+1}$.
  If $p$ is odd then \eqref{p-ineq''} gives 
  $\gamma_p\ell_p+1=a \ell_p+1< \ell_p^2+1<\ell_{p+1}(\ell_p+\sqrt{\ell_p})=\gamma_{p+1}\ell_{p+1}$.

  \medskip

\noindent{\bf Case IV}:
 $a=b=0$.  By definition, for all $p\in\mathbb N$ we have
\begin{equation}\label{eq-IV}\ell_p\geq\ell_1\geq1.\end{equation} For all $p$, \eqref{eq-IV} gives 
$\gamma_p\ell_p+1=\sqrt{\ell_p}+1< \sqrt{\ell_p+\sqrt{\ell_p}}<\gamma_{p+1}\ell_{p+1}$.

\medskip

 \noindent{\bf Case V:} 
 $a=b=\infty$.  By definition, for all $p\in\mathbb N$ we have
\begin{equation}\label{eq-V}\ell_p\geq\ell_1\geq1.\end{equation} For all $p\geq1$, \eqref{eq-V} gives 
$\gamma_p \ell_p+1=\ell_p^2+1<(\ell_{p}+\sqrt{\ell_p})^2=\gamma_{p+1}\ell_{p+1}$.

\medskip

  \noindent{\bf Case VI:} 
$0=a<b=\infty$. By definition, for all $p\in\mathbb N$ we have
\begin{equation}\label{ell-1}\ell_p\geq\ell_1\geq1.\end{equation}
If $p$ is even then \eqref{ell-1} gives $\gamma_p\ell_p+1=\ell_p^2+1<\ell_p(\ell_p+\sqrt{\ell_p})=\gamma_{p+1}\ell_{p+1}.$
   If $p$ is odd then \eqref{ell-1} gives $\gamma_p\ell_p+1=\ell_p/\sqrt{\ell_{p-1}}+1< \ell_{p}(\ell_p+\sqrt{\ell_p})=
\gamma_{p+1}\ell_{p+1}.$
\medskip

To prove (f),
set $\ell_0=\gamma_0=0$ for convenience. Let $c\geq0$. By (c) and (e), for each $j\in\mathbb N$ there is $L_j>0$ such that $\lim_{j\to\infty}L_j=\infty$ and
$L_j(c+\ell_j)\leq e^{\gamma_{j-1}\ell_{j-1}}(e-1)\leq e^{\gamma_j\ell_j}-e^{\gamma_{j-1}\ell_{j-1}}.$
Summing this inequality over all $1\leq j\leq p$ and rearranging the result implies the desired equality.
The proof of lemma~\ref{computational} is complete. \qed

\subsection{Transitivity of the alpha-beta transformations}\label{transitivity} 
The next proposition is a supplement to corollary~\ref{cor1}.
\begin{prop}\label{trans-prop}If $0\leq\alpha<1$ and $\beta\geq2$, then $T_{\alpha,\beta}$ is transitive.\end{prop}
\begin{proof}
Write $T$ for $T_{\alpha,\beta}$.
 Although $T$ is not a homeomorphism, it suffices to show that for any pair $(U,V)$ of non-empty open subsets of $[0,1]$, there is $n\in\mathbb N$ such that $U\cap T^{-n}(V)\neq\emptyset$.

Let $S$ denote the set of discontinuities of $T$. 
Let $U$, $V$ be non-empty open subsets of $[0,1]$.
We construct non-empty open subintervals $A_0,A_1,\ldots$ of $(0,1)$ inductively as follows.
Let $A_0$ be an arbitrary open interval that is contained in $U$ and does not contain two elements of 
$S$. 
Suppose we are given $A_n$ for some $n\geq0$ 
that does not contain two elements of $S$.
If $A_n$ is disjoint from $S$,
then we set $A_{n+1}=T(A_n)$. Otherwise, 
we split $A_n\setminus S=J\bigsqcup K$ so that ${\rm diam}(J)>{\rm diam}(K)$, or else ${\rm diam}(J)={\rm diam}(K)$ and $\inf K\geq\sup J$, and set $A_{n+1}=T(K)$.
Clearly,
if $A_n$ and $A_{n+1}$ are defined then 
${\rm diam}(A_{n+1})\geq(\beta/2){\rm diam}(A_n)$.
If moreover $A_{n+1}=T(A_n)$ then ${\rm diam}(A_{n+1})=\beta\cdot {\rm diam}(T(A_n))$.

If $\beta>2$, then ${\rm diam}(A_n)$ keeps growing exponentially as $n$ increases. So,
we must reach $n\geq1$ such that $A_0,\ldots,A_n$ are defined and
$A_n$ contains two elements of $S$.
Then $A_n$ contains $I_j$ for some $1\leq j\leq m-2.$
It follows that $A_0$ contains an open interval that is mapped by  $T^{n+1}$ homeomorphically onto $(0,1)$. We are done.

Suppose $\beta=2$. Then $S$ consists of two elements. If $n\geq1$, $A_0,\ldots,A_n$ are defined and $A_n\supset S$, then $A_n$ contains $I_1$. It follows that $A_0$ contains an open interval that is mapped by $T^{n+1}$ homeomorphically onto $(0,1)$.
Therefore,
it is left to consider the case where $A_n$ is defined for all $n\geq0$.
Since $\beta>1$, $\{n\geq0\colon A_{n+1}=T(A_n)\}$ is a finite set.
Since $S=\{(1-\alpha)/2,(2-\alpha)/2\}$,
$T(I_0)=[\alpha,1)$, $T(I_2)=[0,\alpha)$,
and $1-\alpha$ is the unique fixed point of $T_{\alpha,2}$,
 there is $n\geq0$ such that one of the following three cases occurs:

\begin{itemize}
\item[(i)] $A_n=T(I_0)$ and $A_{n+1}=T(I_2)$;

\item[(ii)] $A_n=T(I_2)$ and $A_{n+1}=T(I_0)$;

\item[(iii)] $1-\alpha\in A_n$.
\end{itemize}
Since $T(I_0)\cup T(I_2)=[0,1)$, 
if (i) or (ii) occurs then we are done.
If (iii) occurs, then there is $n'\geq n+1$ such that $A_0$ contains an open interval that is mapped by $T^{n'}$ diffeomorphically onto $(0,1)$. 
The proof of proposition~\ref{trans-prop} is complete.
\end{proof}

\subsection*{Acknowledgments} 
I thank anonymous referees and the editor for their careful readings of the manuscript and giving valuable suggestions.
Part of this paper was written while I was staying in University of Maryland. I thank James A. Yorke for his hospitality during this visit.
This research was partially supported by the JSPS KAKENHI 25K21999.

      \bibliographystyle{amsplain}

\end{document}